%% file: dbm.tex
\documentclass[onefignum,onetabnum]{siamart190516}
\input{dbm_shared}

\ifpdf
\hypersetup{
  pdftitle={Computing multiple solutions of topology optimization problems},
  pdfauthor={I. P. A. Papadopoulos, P. E. Farrell, and T. M. Surowiec}
}
\fi

\begin{document}

\maketitle

\begin{abstract}
Topology optimization problems often support multiple local minima due to a lack of convexity. Typically, gradient-based techniques  combined with continuation in model parameters are used to promote convergence to more optimal solutions; however, these methods can fail even in the simplest cases. In this paper, we present an algorithm to perform a systematic exploratory search for the solutions of the optimization problem via second-order methods without a good initial guess. The algorithm combines the techniques of deflation, barrier methods and primal-dual active set solvers in a novel way. We demonstrate this approach on several numerical examples, observe mesh-independence in certain cases and show that multiple distinct local minima can be recovered.
\end{abstract}



\section{Introduction}
Topology optimization has become popular as an effective technique in structural and additive manufacturing, and has found uses in architecture, medicine and material science \cite{Adam2019, Jang2008, Liu2018}. The objective is to find the optimal distribution of a fluid or material within a given domain that minimizes a problem-specific cost functional. In contrast to shape optimization, the topology of the structure does not need to be chosen a priori. 

There are several mathematical parametrizations for the topology of a material including density approaches \cite{Bendsoe1989, Bendsoe2004, Borrvall2003, Mlejnek1992} and level set methods \cite{Allaire2002, Allaire2004, Wang2003}; these can be optimized by a variety of strategies such as topological derivatives \cite{Sokolowski1999}, evolutionary methods \cite{Xie1993}, the method of moving asymptotes \cite{Svanberg1987}, and barrier methods \cite{Evgrafov2014,Hoppe2002,Maar2000,Rojas-Labanda2015}. We choose to represent our topology with the density approach. This introduces a function, denoted $\rho$, that represents the material distribution over the given domain. Ideally we would find an optimizing material distribution $\rho:\Omega \to \{0,1\}$ indicating presence or absence of material. However, this is numerically intractable in general and we therefore consider densities $\rho:\Omega \to [0,1]$ in order to exploit continuous optimization techniques. The model is then regularized to favor solutions where $\rho$ is close to zero or one. 

Due to the nonlinear relation between $\rho$ and the solution of the underlying physical system, multiple local minima can occur even in problems with a linear governing partial differential equation (PDE). For example, minimizing the power dissipation of a fluid governed by the Stokes equations flowing through a pipe can give rise to distinct pipe configurations that locally minimize the power lost to dissipation \cite[Sec.\ 4.5]{Borrvall2003}. Currently, the main technique to address this is the use of \textit{continuation methods} to promote convergence to better local minima. However, Stolpe and Svanberg \cite{Stolpe2001} have provided elementary examples where these continuation methods fail. For example, a solid isotropic material with penalization (SIMP) formulation \cite{Bendsoe2004} of the compliance minimization of a six-bar truss can be reduced to the optimization problem \cite[Sec.\ 3.1]{Stolpe2001},
\begin{align*}
&\min_{(x_1, x_2) \in \mathbb R^2}
\left( 
\max\{
\frac{8 \beta_t}{x_1^{p_s} + 5x_2^{p_s}} + \frac{2 \beta_t}{5x_1^{p_s} + x_2^{p_s}}, \frac{8}{x_1^{p_s} + 5x_2^{p_s}} + \frac{18}{5x_1^{p_s} + x_2^{p_s}}
\}\right) \\ 
&\indent \text{such that} \;\; x_1 + x_2 = 1,\quad
0 \le x_1,x_2 \le 1.
\end{align*}
Here $p_s$ denotes the SIMP continuation parameter and $\beta_t = 2(1-\nu_t^2)/E$, where $\nu_t$ is the Poisson ratio and $E$ is the modulus of elasticity. SIMP is used to penalize solutions that are not either zero or one and is further discussed in \cref{sec:compliance}. A typical strategy is to find a minimizer to the optimization problem at $p_s =1$, and  then at each continuation step use the previous solution as initial guess for the next value of $p_s$. In this case, suppose we fix $\beta_t = 2.6$. A poor starting guess for $p_s=1$ can converge to the local minimum $x = (0.5,0.5)$. Then even as $p_s \to \infty$, the continuation method will always return $x = (0.5,0.5)$ and will not converge to the true global solution, $x = (0,1)$. 

The calculation of multiple stationary points is important because iterative methods often give no guarantee whether they converge to a local or global minimum. By finding multiple stationary points, one is able to choose the best available, in a postprocessing step. Furthermore, an iterative method may converge to a stationary point which is undesirable due to manufacturing or aesthetic reasons; thus industrial applications can benefit from having a choice of multiple locally optimal configurations \cite{Doubrovski2011}.

In this paper we formulate an algorithm, which we call the \textit{deflated barrier method}, for finding multiple stationary points of topology optimization problems and present several large-scale numerical examples arising from the finite element discretization of PDEs. An example we consider is the topology optimization of the power dissipation of fluid flow governed by the incompressible Navier--Stokes equations on a rectangular domain with five small decagonal holes. We discover 42 stationary points of this optimization problem with the deflated barrier method. The material distribution of these solutions are shown in \cref{fig:fiveholes-navier}. 

The deflated barrier method is a combination of \textit{deflation} \cite{Brown1971, Farrell2015, Farrell2019}, \textit{barrier methods} \cite{Fiacco1990, Forsgren2002, Frisch1955, Schiela2007, Schiela2008,Ulbrich2009, Weiser2008}, \textit{primal-dual active set solvers} \cite{Benson2003, HintermullerIto2003} and \textit{predictor-corrector methods} \cite{Seydel2010}. The combination of primal-dual active set solvers, barrier and deflation methods in the manner proposed is novel. The combination does not suffer the poor behavior that barrier methods traditionally exhibit as the barrier parameter approaches zero. In fact, in our numerical examples, the combination performs better than the optimize-then-discretize formulation of the primal-dual interior point method where Newton--Kantorovich iterates are used to solve the subproblems, either approximately or exactly. The predictor-corrector method is also adapted for use with box-constrained variables to ensure the predictor is feasible. The main contribution of this work is an algorithm to robustly determine multiple solutions to nonconvex, inequality and box-constrained infinite-dimensional optimization problems starting from poor initial guesses. 

Other approaches to computing multiple solutions of topology optimization
problems are possible. Zhang and Norato \cite{Zhang2018} apply the
\emph{tunneling} method \cite{Levy1984} to these problems, adapting the method
of moving asymptotes. Tunneling proceeds by finding a single minimum, then
looking for other controls that yield the same functional value (attempting to
tunnel into other basins) by solving an auxiliary equation. Deflation is used in
the tunneling phase to ensure that the Gauss--Newton procedure applied to the
tunneling functional does not converge to the current state.

The outline of the paper is as follows. In \cref{sec:topoptformulation} we formulate some topology optimization problems for pipe design and structural compliance. The deflated barrier method is described in \cref{sec:alg}. Several examples of topology optimization problems are given in \cref{sec:numerical}, where we discover multiple solutions for Navier--Stokes flow, Stokes flow, and structural compliance, and consider the performance of our algorithm. In \cref{sec:conclusions} we outline our conclusions. A result concerning the equivalence of Hinterm\"uller et al.'s primal-dual active set strategy \cite{HintermullerIto2003} and Benson and Munson's reduced space active-set strategy \cite{Benson2003} is given in \cref{sec:BMsolver}. In \cref{sec:tangentpredictiontask} we describe our novel feasible tangent prediction method. 
\begin{figure}
\centering
\includegraphics[width = 0.15\textwidth]{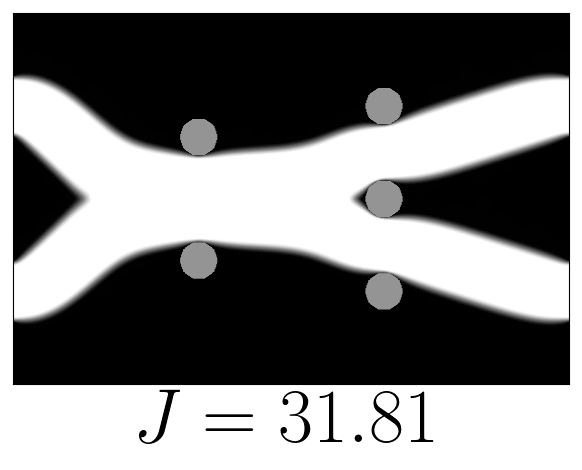}
\includegraphics[width = 0.15\textwidth]{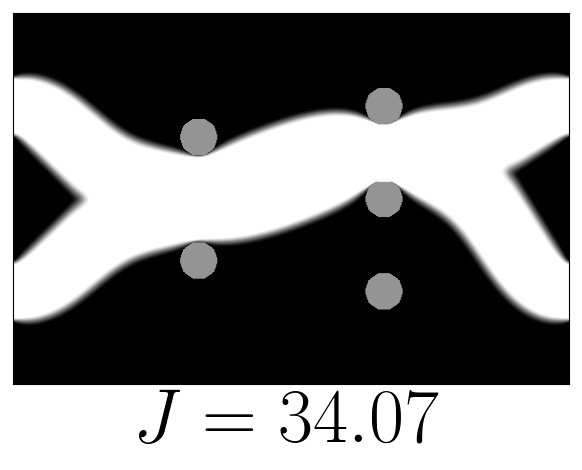}
\includegraphics[width = 0.15\textwidth]{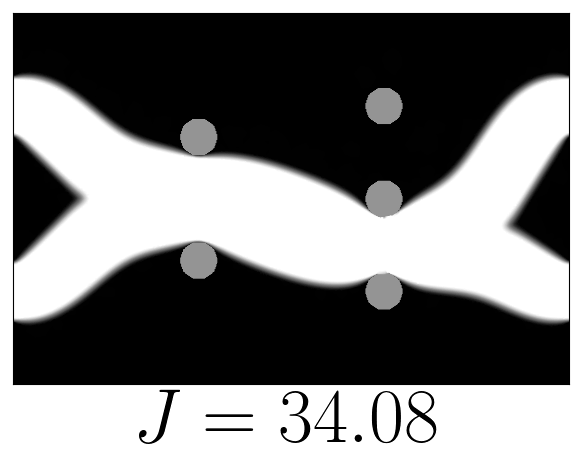}
\includegraphics[width = 0.15\textwidth]{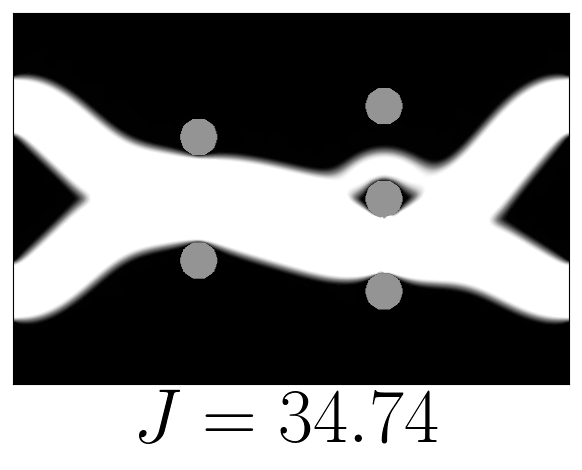}
\includegraphics[width = 0.15\textwidth]{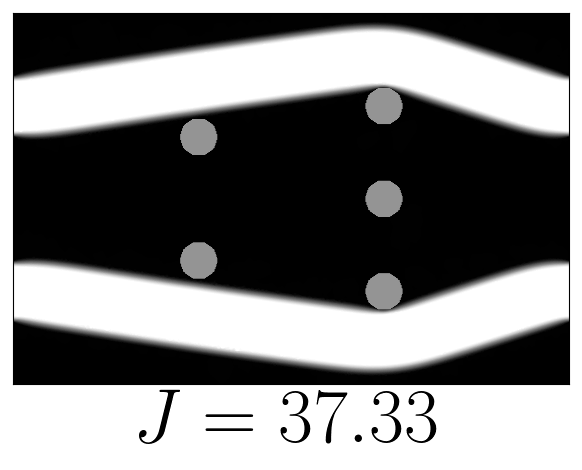}
\includegraphics[width = 0.15\textwidth]{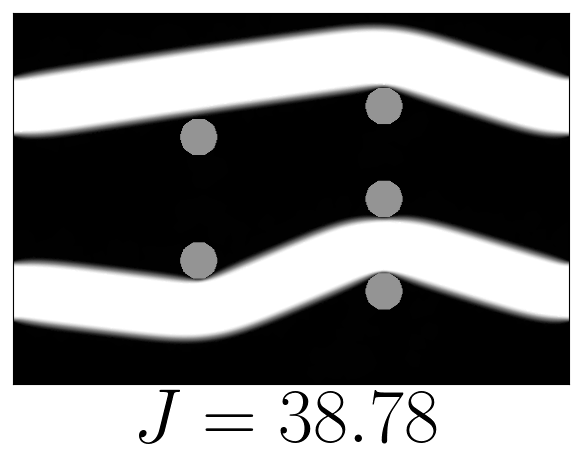}
\includegraphics[width = 0.15\textwidth]{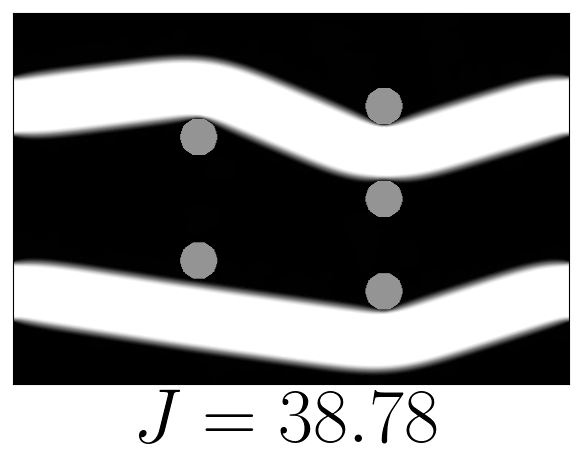}
\includegraphics[width = 0.15\textwidth]{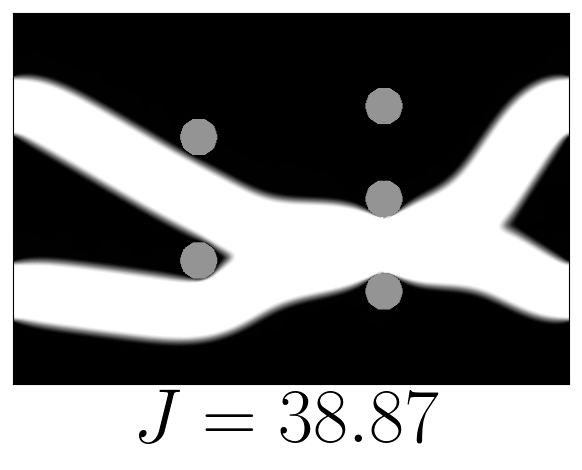}
\includegraphics[width = 0.15\textwidth]{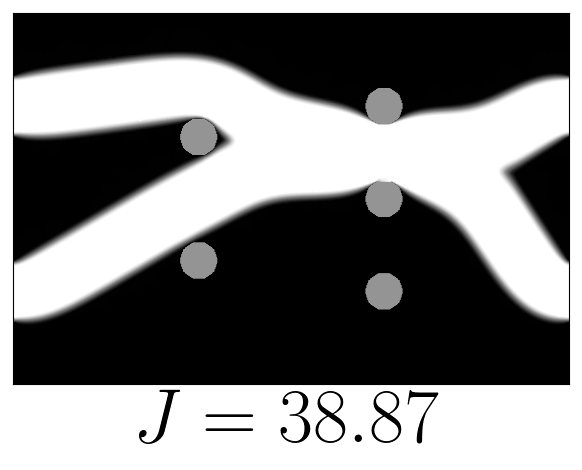}
\includegraphics[width = 0.15\textwidth]{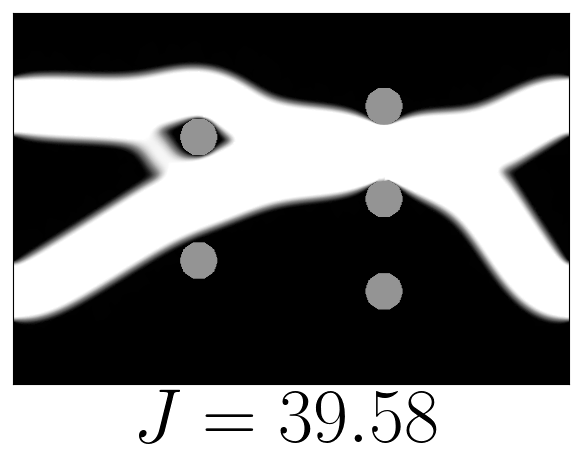}
\includegraphics[width = 0.15\textwidth]{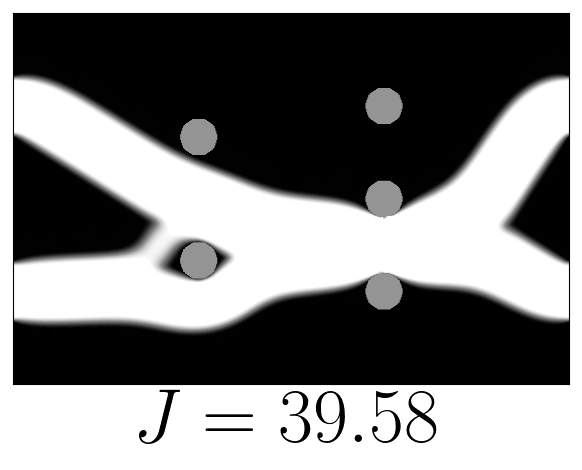}
\includegraphics[width = 0.15\textwidth]{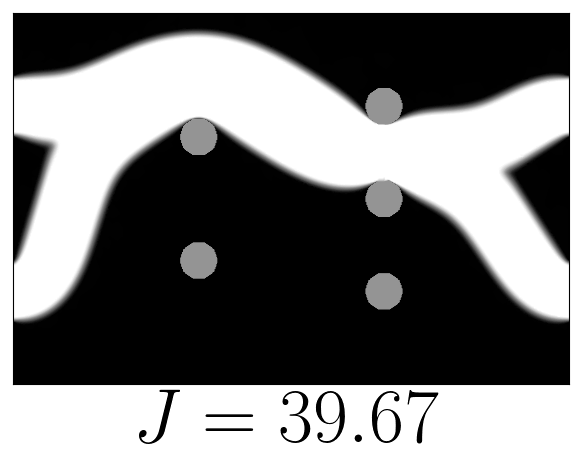}
\includegraphics[width = 0.15\textwidth]{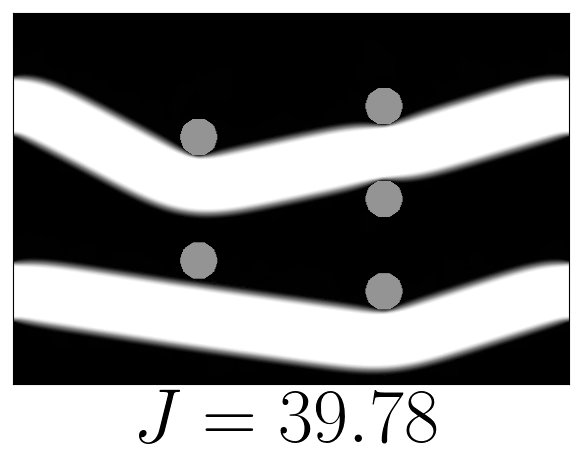}
\includegraphics[width = 0.15\textwidth]{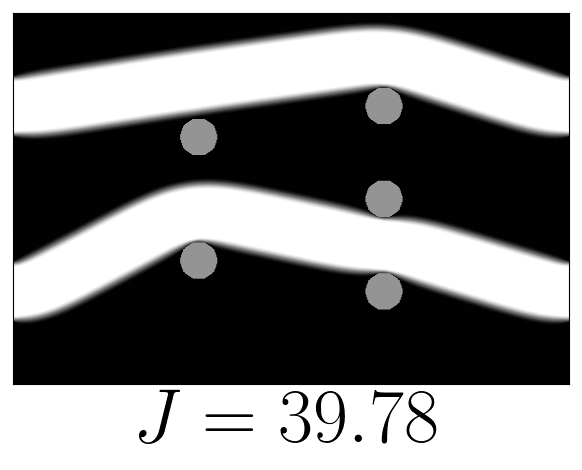}
\includegraphics[width = 0.15\textwidth]{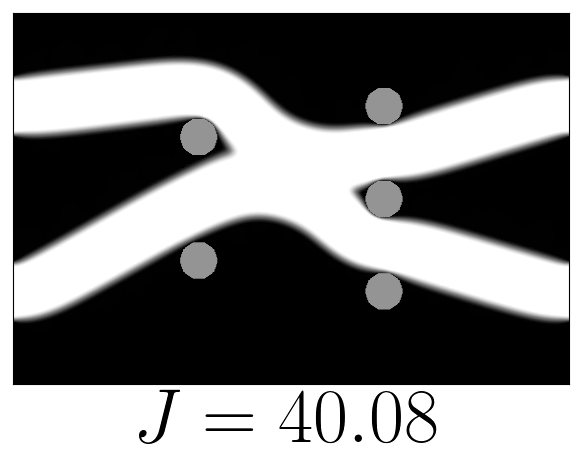}
\includegraphics[width = 0.15\textwidth]{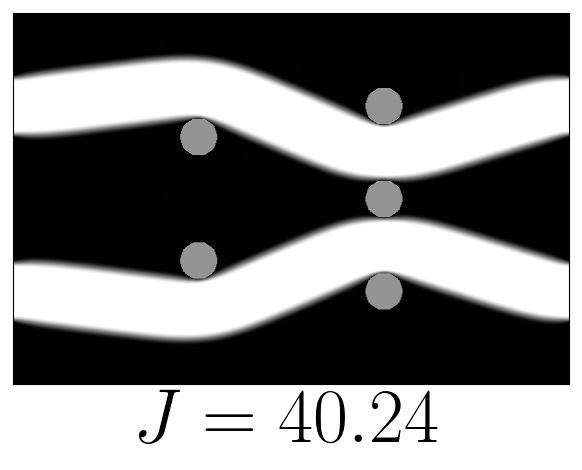}
\includegraphics[width = 0.15\textwidth]{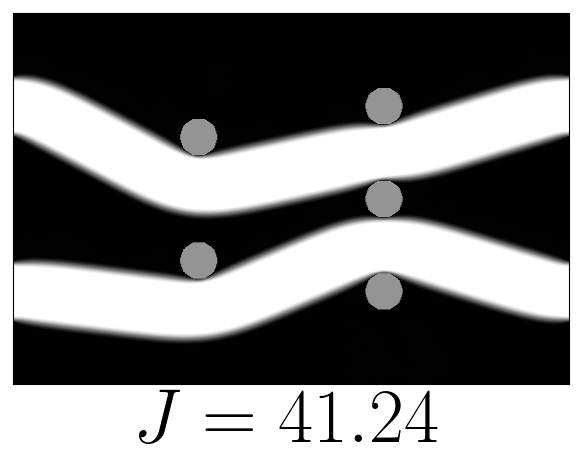}
\includegraphics[width = 0.15\textwidth]{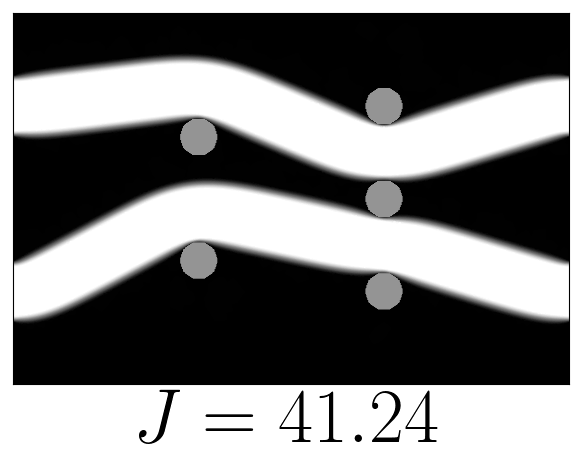}
\includegraphics[width = 0.15\textwidth]{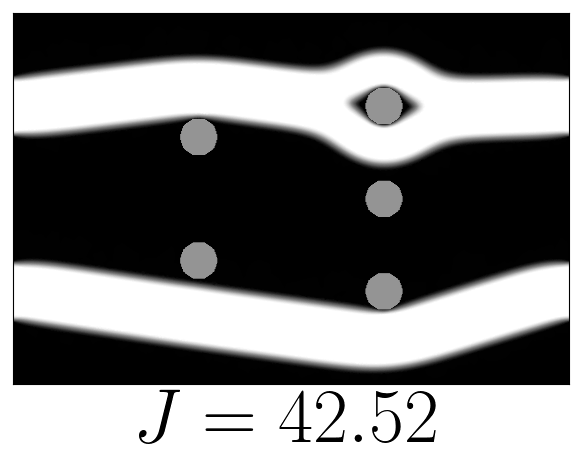}
\includegraphics[width = 0.15\textwidth]{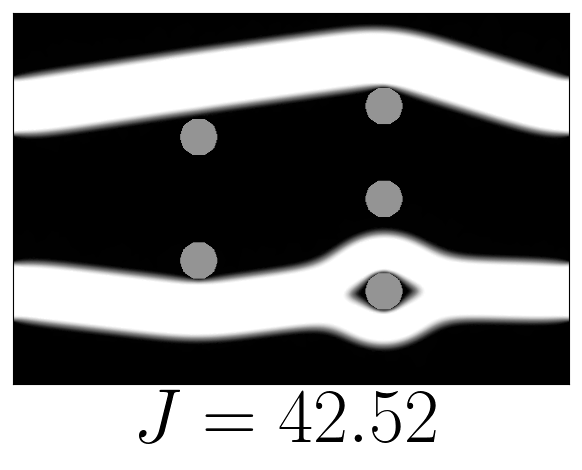}
\includegraphics[width = 0.15\textwidth]{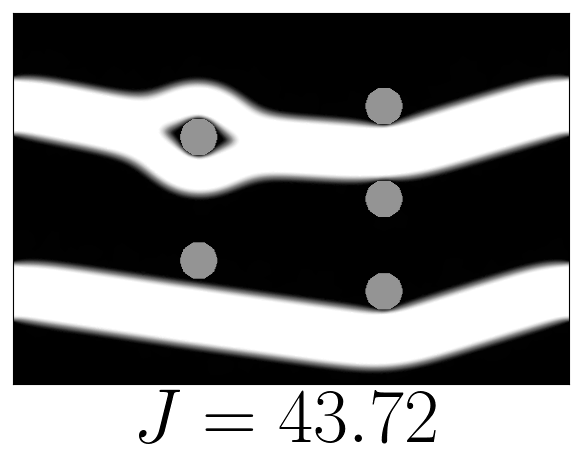}
\includegraphics[width = 0.15\textwidth]{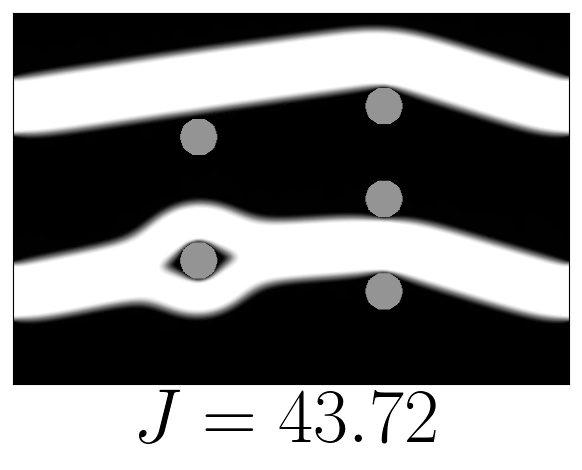}
\includegraphics[width = 0.15\textwidth]{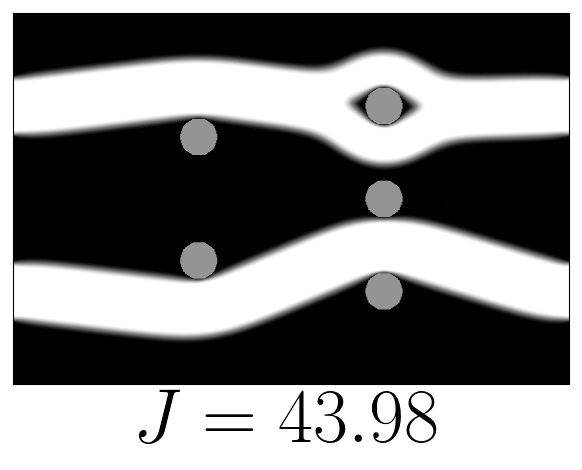}
\includegraphics[width = 0.15\textwidth]{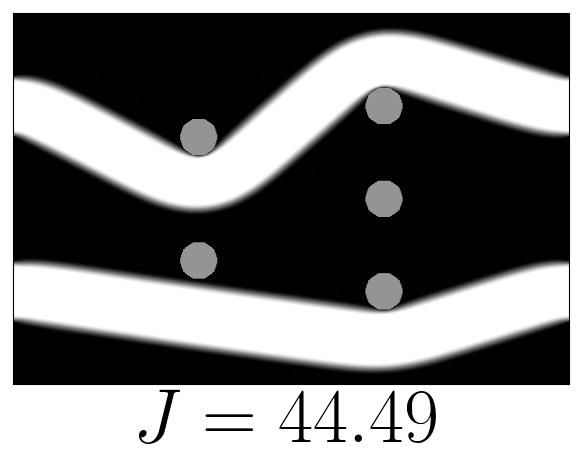}
\includegraphics[width = 0.15\textwidth]{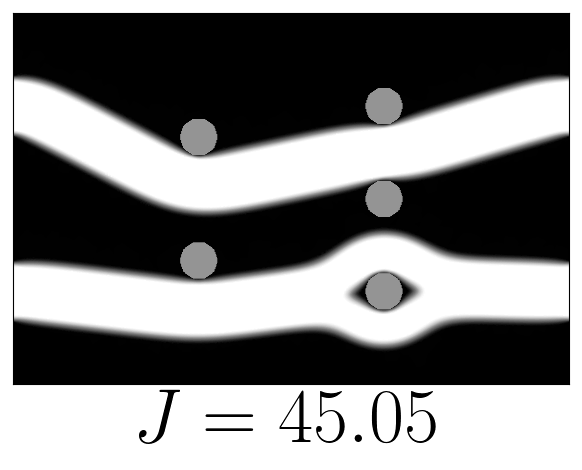}
\includegraphics[width = 0.15\textwidth]{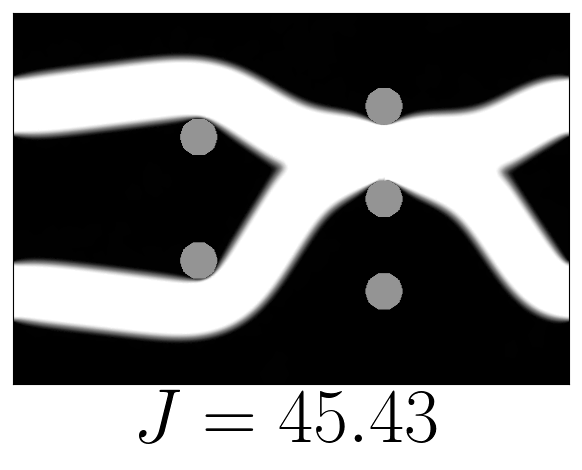}
\includegraphics[width = 0.15\textwidth]{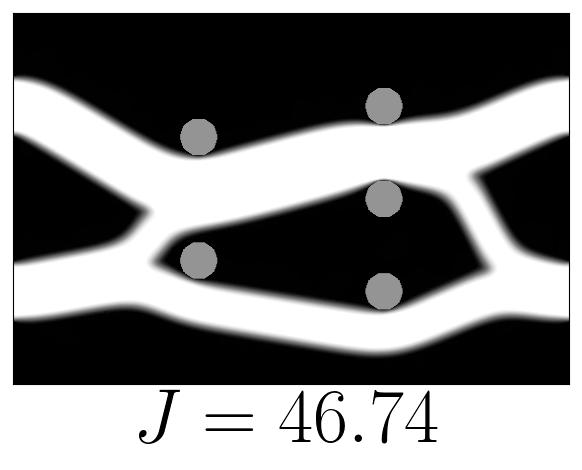}
\includegraphics[width = 0.15\textwidth]{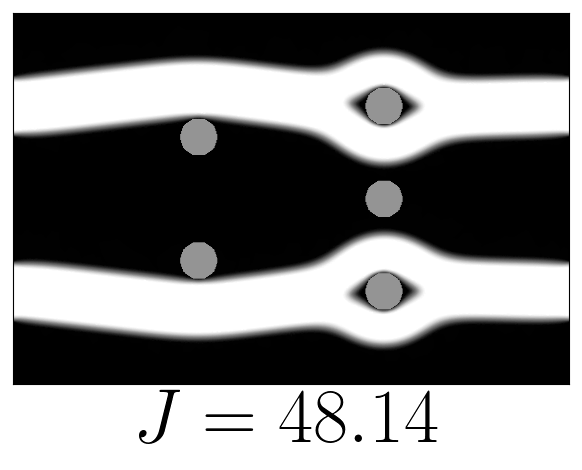}
\includegraphics[width = 0.15\textwidth]{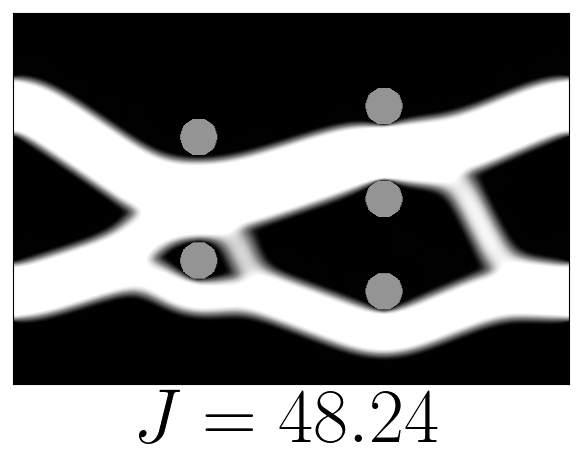}
\includegraphics[width = 0.15\textwidth]{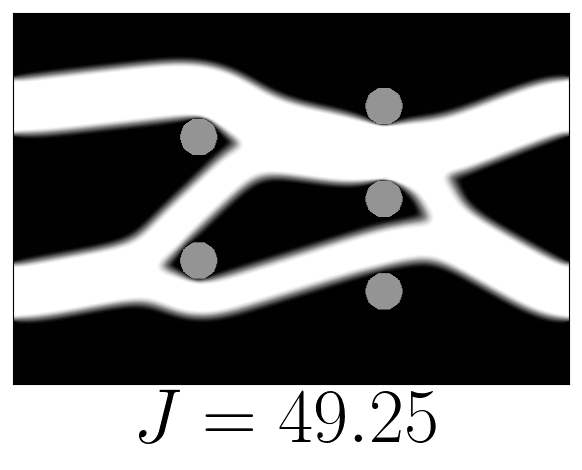}
\includegraphics[width = 0.15\textwidth]{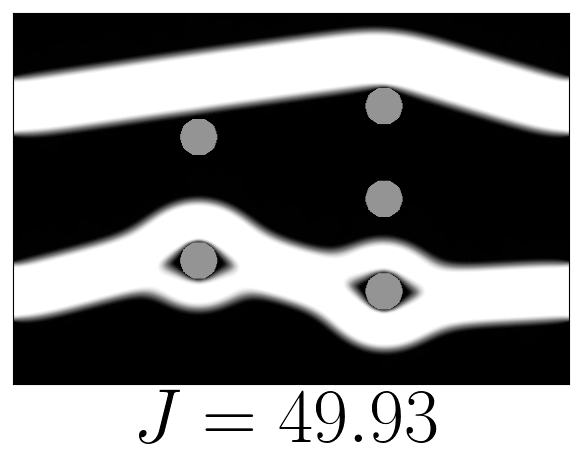}
\includegraphics[width = 0.15\textwidth]{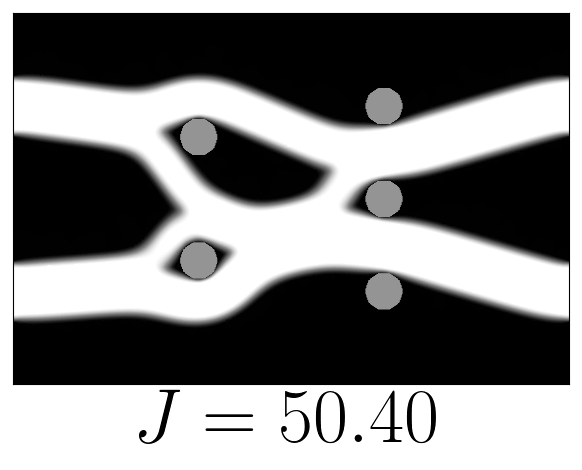}
\includegraphics[width = 0.15\textwidth]{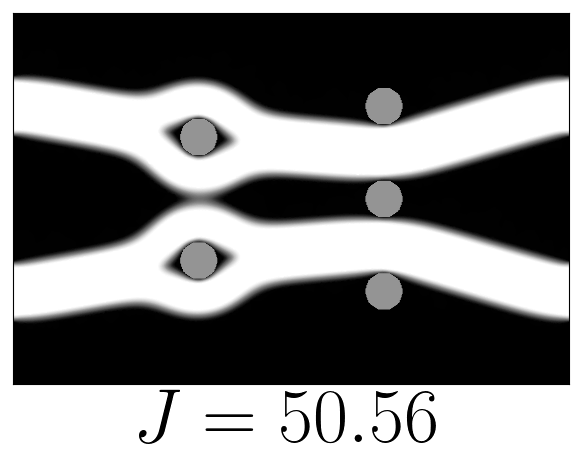}
\includegraphics[width = 0.15\textwidth]{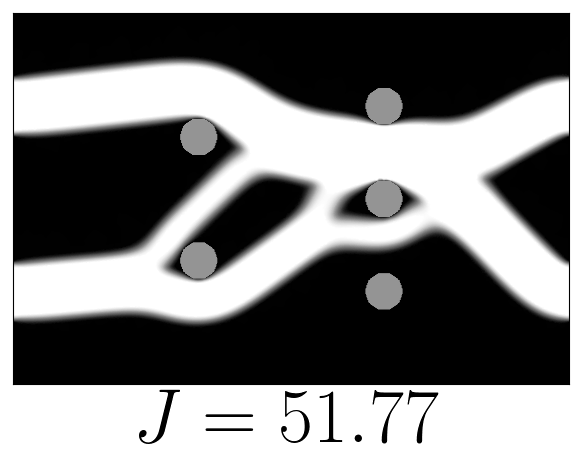}
\includegraphics[width = 0.15\textwidth]{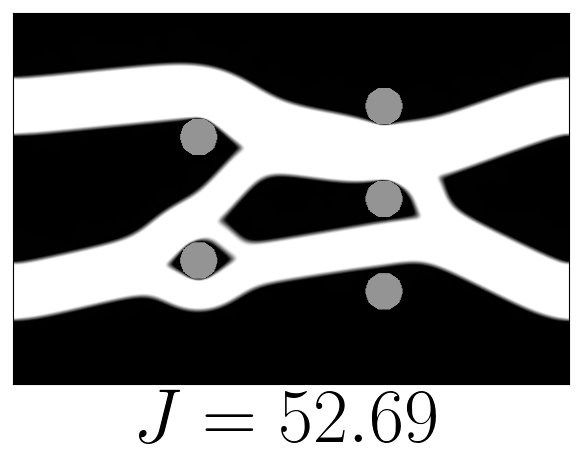}
\includegraphics[width = 0.15\textwidth]{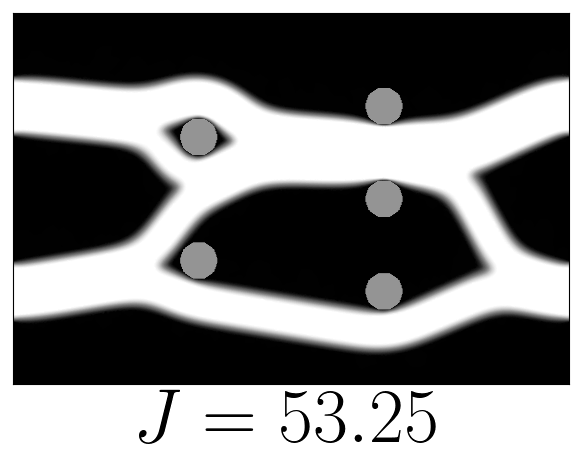}
\includegraphics[width = 0.15\textwidth]{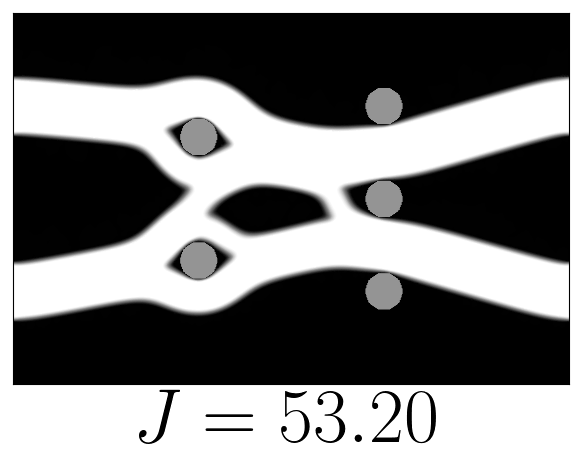}
\includegraphics[width = 0.15\textwidth]{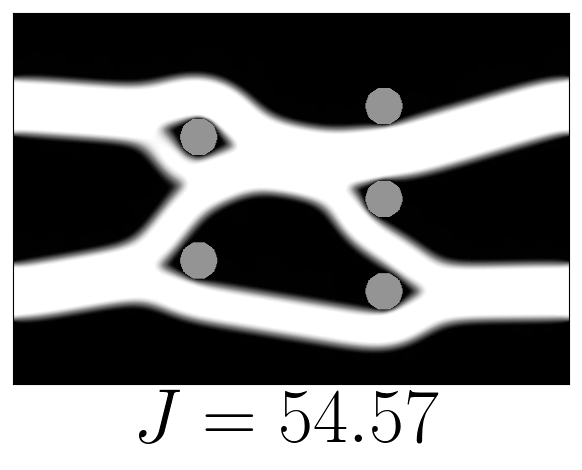}
\includegraphics[width = 0.15\textwidth]{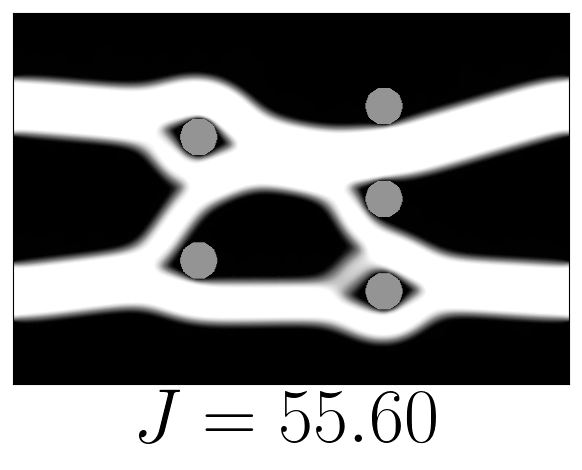}
\includegraphics[width = 0.15\textwidth]{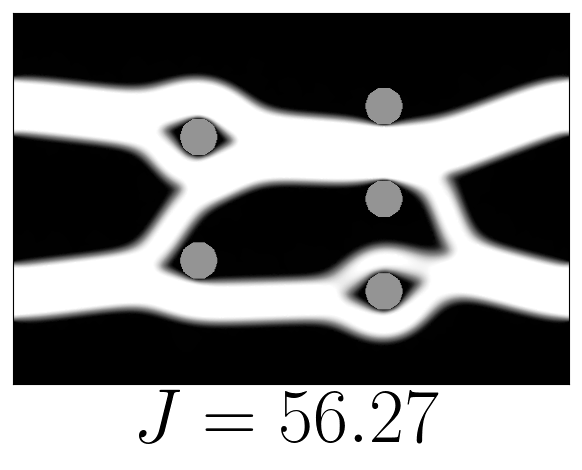}
\includegraphics[width = 0.15\textwidth]{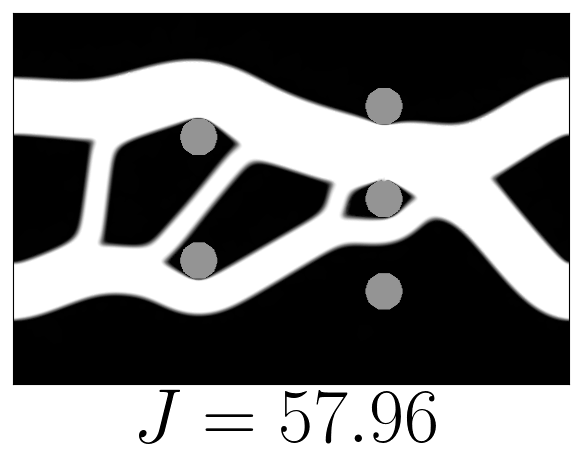}
\includegraphics[width = 0.15\textwidth]{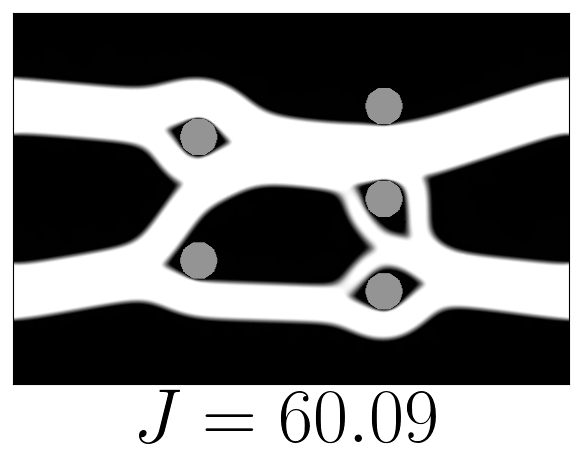}
\caption{The material distribution of 42 stationary points of the five-holes double-pipe optimization problem as discovered by the deflated barrier method, and their associated energies $J$. The fluid flow is governed by the incompressible Navier--Stokes equations. The formulation of the problem is described in \cref{sec:fiveholesdoublepipe}. Black corresponds to a value of $\rho = 0$, white corresponds to a value of $\rho = 1$, and the gray regions are the five small holes.}
\label{fig:fiveholes-navier}
\end{figure}
\newpage
\section{Topology optimization formulations}
\label{sec:topoptformulation}
\subsection{Topology optimization of Stokes flow}
\label{sec:borrvall}
We consider the formulation of the topology optimization of fluids proposed in the pioneering work of Borrvall and Petersson \cite{Borrvall2003}. They derive a `generalized Stokes problem' incorporating a material distribution variable which has a value of one where fluid is present and zero where there is void. The derived optimization problem requires no further regularization for well-posedness, in contrast to structural topology optimization. The optimization problem supports (not necessarily unique) local minima. 

The topology optimization problem of Borrvall and Petersson is
\begin{align}
\min_{(\vect{u},\rho) \in H^1_{\vect{g},\text{div}}(\Omega)^d \times C_\gamma} J(\vect{u},\rho) :=\frac{1}{2} \int_\Omega \left(\alpha(\rho) |\vect{u}|^2 + \nu |\nabla \vect{u}|^2 - 2\vect{f} \cdot \vect{u}\right)  \text{d}x,\label{borrvallmin} \tag{BP}
\end{align}
where $\vect{u}$ denotes the velocity of the fluid, $\rho$ is the material distribution of the fluid and
\begin{align*}
H^1_{\vect{g}}(\Omega)^d &:= \{\vect{v} \in \underbrace{H^1(\Omega) \times \dots \times  H^1(\Omega)}_{d \; \text{times}}: \vect{v}|_{\partial \Omega} = \vect{g} \},\\
H^1_{\vect{g},\text{div}}(\Omega)^d &:= \{\vect{v} \in H^1_{\vect{g}}(\Omega)^d : \text{div}(\vect{v}) = 0 \; \text{a.e.\ in} \; \Omega \},\\
C_\gamma &:= \left\{ \eta \in L^\infty(\Omega) : 0 \leq \eta \leq 1 \; \text{a.e.}, \;\; \int_\Omega \eta \; \text{d}x \leq \gamma |\Omega|, \; \gamma \in (0,1) \right \}.
\end{align*}
In this work, $H^1(\Omega)$ denotes the Sobolev space $W^{1,2}(\Omega)$ and $L^\infty(\Omega)$ denotes the vector space of essentially bounded measurable functions equipped with the essential supremum norm. Furthermore, $\Omega \subset \mathbb{R}^d$ is a Lipschitz domain with dimension $d=2$ or $d=3$, $\vect{f} \in L^2(\Omega)^d$ is a body force and $\nu > 0$ is the (constant) viscosity. The restriction, $|_{\partial \Omega}$, is to be understood in the boundary trace sense. Moreover, the boundary data $\vect{g} \in H^{1/2}(\partial \Omega)^d$ and $\vect{g} = \vect{0} $ on $\Gamma \subset \partial \Omega$, with $\mathcal{H}^{d-1}(\Gamma)>0$, i.e.\ $\Gamma$ has nonzero Hausdorff measure on the boundary. Mixed boundary conditions are discussed in \cref{sec:neumann-double-pipe}. Here, $\alpha$ is the inverse permeability, modeling the influence of the material distribution on the flow. For values of $\rho$ close to one, $\alpha(\rho)$ is small, permitting fluid flow; for small values of $\rho$, $\alpha(\rho)$ is very large, restricting fluid flow. The function $\alpha$ satisfies the following properties:
\begin{enumerate}[label=({A}\arabic*)]
\item $\alpha: [0,1] \to [\underline{\alpha}, \overline{\alpha}]$ with $0 \leq \underline{\alpha} < \overline{\alpha} < \infty$;
\label{alpha1}
\item $\alpha$ is convex and monotonically decreasing;
\item $\alpha(0) = \overline{\alpha}$ and $\alpha(1) = \underline{\alpha}$,
\label{alpha3}
\end{enumerate}
generating a superposition operator also denoted $\alpha: C_\gamma \to L^\infty(\Omega; [\underline{\alpha},\overline{\alpha}])$. Typically, in the literature $\alpha$ takes the form \cite{Borrvall2003, Evgrafov2014}
\begin{align}
\alpha(\rho) = \bar{\alpha}\left( 1 - \frac{\rho(q+1)}{\rho+q}\right), \label{eq:alphachoice}
\end{align}
where $q>0$ is a penalty parameter, so that $\lim_{q \to \infty} \alpha(\rho) = \bar{\alpha}(1-\rho)$. The objective functional \cref{borrvallmin} can be interpreted as the total potential power of the flow. The first and second terms in the integral measure the energy lost by the flow through the porous medium and the energy lost due to viscous dissipation, respectively. The third term attempts to maximize the flow velocities at the applied body force. \cref{borrvallmin} is discussed in further detail by  Borrvall and Petersson \cite{Borrvall2003}.
\begin{remark}
The integral in \cref{borrvallmin} is well defined. {Indeed,} since $\alpha$ is assumed to be convex, it is {Borel measurable; also since} $\rho \in C_\gamma$ is Lebesgue measurable, the composition {$\alpha(\rho) : \Omega \to [\underline{\alpha}, \overline{\alpha}]$} is Lebesgue measurable. 
\end{remark}
\begin{theorem} 
\label{th:BPexistence}
\cite[Th.\ 3.1]{Borrvall2003} Suppose that $\Omega \subset \mathbb{R}^d$ is a Lipschitz domain, with $d=2$ or $d=3$ and $\alpha$ {satisfies properties \labelcref{alpha1}--\labelcref{alpha3}. Then there exists a pair} $(\vect{u}, \rho) \in H^1_{\vect{g},\mathrm{div}}(\Omega) \times C_\gamma$ that minimizes $J$, as defined in \cref{borrvallmin}. 
\end{theorem}
Due to the lack of strict convexity in \cref{borrvallmin}, a minimizing pair is not necessarily unique.

\subsection{Construction of the barrier functional}
In this subsection we formulate a barrier functional with an enlarged feasible set that will be employed by our algorithm to find multiple solutions of the Borrvall--Petersson optimization problem. 

We first consider the volume constraint. This constraint is typically modeled as an inequality constraint. However, as we show below, this constraint is active at an optimal solution, and so we may also apply it as an equality constraint. To the best of our knowledge, the following result is novel.
\begin{proposition}
If the pair $(\vect{u}_*, \rho_*)$ is an isolated local or global minimizer of $J$ as defined in \cref{borrvallmin} and $\gamma <1$, then $\int_\Omega \rho_* \; \mathrm{d}x = \gamma|\Omega| $. 
\end{proposition}
\begin{proof}[Proof by contradiction] 
Suppose there exists a pair $(\vect{u}_*, \rho_*) \in H^1_{\vect{g},\text{div}}(\Omega)^d \times C_\gamma$ that is an isolated local or global minimizer of $J(\vect{u},\rho)$ such that $V:= \int_\Omega \rho_* \; \text{d}x < \gamma|\Omega|$.  By the definition of an isolated local minimizer, there exists an $r > 0$ such that for any $(\vect{v}, \eta)$ that satisfies,
\begin{align*}
\| \vect{u}_* - \vect{v}\|_{H^1(\Omega)} + \| \rho_* - \eta \|_{L^\infty(\Omega)} \leq r
\end{align*}
then $J(\vect{u}_*, \rho_*) < J(\vect{v}, \eta)$. Then for any function $\delta \rho \in C_\gamma$ such that 
\begin{align}
0 < \|  \delta \rho \|_{L^1(\Omega)} &\leq (\gamma |\Omega| - V), \label{eq:l1check}\\
0 < \|  \delta \rho \|_{L^\infty(\Omega)} &\leq r, \label{eq:lqcheck}\\
0 \leq \rho_* + \delta \rho &\leq 1, \label{eq:vicheck}
\end{align}
we have that $\rho_* + \delta \rho \in C_\gamma$ from \cref{eq:l1check} and \cref{eq:vicheck} and $\rho_* + \delta \rho$ lies in the $L^\infty $-$r$-neighborhood of $\rho_*$ from \cref{eq:lqcheck}. Such a $\delta \rho$ exists, for example,
\begin{align*}
\delta \rho = c(1 -\rho_*), \;\; \text{where} \;  c = \text{min}\left\{ \frac{r}{\| 1 - \rho_* \|_{L^\infty(\Omega)}}, \frac{\gamma |\Omega|- V}{|\Omega| - V} \right\}.
\end{align*}
We see that $c > 0$ since $r > 0$ and $V < \gamma |\Omega| < |\Omega|$. Furthermore $\delta \rho$ satisfies \cref{eq:l1check}--\cref{eq:vicheck} since, 
\begin{align*}
 \|  \delta \rho \|_{L^1(\Omega)}  &= c \int_\Omega (1-\rho_*) \text{d}x \leq c (|\Omega| - V) \leq \gamma |\Omega| - V,\\
  \|  \delta \rho \|_{L^\infty(\Omega)} &\leq c \| 1 - \rho_*\|_{L^\infty(\Omega)} \leq r,\\
  0 &\leq \rho_* + \delta \rho = \rho_* + c( 1 - \rho_*) \leq \rho_* + 1 - \rho_* \leq 1.
 \end{align*}
Since $\alpha(\cdot)$ is monotonically decreasing and $\rho_*$ and $\delta \rho$ are non-negative and not equal to zero, then $\alpha(\rho_* + \delta \rho ) \leq \alpha(\rho_*) $ a.e.\ and hence $J(\vect{u}_*, \rho_* + \delta \rho) \leq J(\vect{u}_*, \rho_*)$. 
\end{proof}
Given we can tighten the inequality volume constraint to an equality volume constraint, we now define the Lagrangian and the enlarged feasible-set barrier functional, respectively, as:
\begin{align}
L(\vect{u}, \rho, p, p_0, \lambda) &:= J(\vect{u},\rho) - \int_\Omega p \; \text{div}(\vect{u})  \text{d}x - \int_\Omega p_0 p \; \text{d}x -  \int_\Omega \lambda (\gamma - \rho) \text{d}x;
\label{eq:Lagrangian}\\
\begin{split}
L^{\epsilon_{\text{log}}}_\mu(\vect{u}, \rho, p, p_0, \lambda)& := L(\vect{u}, \rho, p, p_0, \lambda)\\
&\indent - \mu \int_\Omega (\log(-\epsilon_{\text{log}} + \rho) + \log(1+\epsilon_{\text{log}}-\rho)) \text{d}x, 
\end{split}
\label{eq:enlargedbarrier}
\end{align}
where $p \in L^2(\Omega)$ denotes the pressure, $\lambda$ is the Lagrange multiplier for the volume constraint, $p_0 \in \mathbb{R}$ is the Lagrange multiplier to fix the integral of the pressure, $0 \leq \epsilon_{\text{log}} \ll 1$ and $\mu \geq 0$, where $\mu$ is the barrier parameter.  

The classical barrier functional is given by $L^0_\mu$. The role of $ \epsilon_{\text{log}}$ is to enlarge the feasible region permitted by the barrier terms. In the deflated barrier method we do \emph{not} use the barrier terms to enforce the box-constraints on $\rho$, but rather to perform continuation in the barrier parameter to follow a central path. This provides robust convergence and offers an opportunity to find other solutions of the optimization problem, as explained in \cref{sec:alg}.

We note that the Euler--Lagrange equation of $J(\vect{u},\rho)$ with respect to $\vect{u}$ satisfies the generalized Stokes momentum equation formulated by Borrvall and Petersson \cite[Eq.\ 12]{Borrvall2003}. Hence, we are only required to enforce the incompressibility and volume constraints. In the case where we wish to minimize the power dissipation of a fluid flow governed by a generalized Navier--Stokes momentum equation, we are required to introduce three extra Lagrange multipliers, as done in \cref{sec:fiveholesdoublepipe}.

\subsection{Topology optimization of the compliance of elastic structures}
\label{sec:compliance}
A significant portion of the topology optimization literature focuses on minimizing the compliance of a structure, such as a Messerschmitt--B\"olkow--Blohm (MBB) beam or a cantilever. Compliance problems involve finding the optimal topology of a structure obeying a volume constraint within a specified domain that minimizes the displacement of the structure under a body or boundary force. For simplicity we consider structures that obey linear elasticity. The optimization problem we consider is posed as follows,
\begin{align}
\min_{(\vect{u}, \rho) \in H^1_{\Gamma_D}(\Omega)^d \times C_\gamma} J(\vect{u},\rho) := \int_{\Gamma_N} \vect{f} \cdot \vect{u} \; \text{d}s \label{complianceopt} \tag{C}
\end{align}
such that,
\begin{align*}
-\text{div}\left( \sigma \right) &= 0 && \text{in} \; \Omega,\\
\sigma &= k(\rho)\left[2 \mu_l  \varepsilon(\vect{u}) + \lambda_l \text{tr}(\varepsilon(\vect{u})) \mathbb{I} \right]&& \text{in} \; \Omega,\\ 
\sigma \vect{n} & = \vect{f} \; \text{on} \; \Gamma_N, \; \; 0 \leq \rho \leq 1 \; \text{a.e.\ in} \; \Omega, \;\; \text{and} \;\;
\int_\Omega \rho \; \text{d}x = \gamma |\Omega| ,
\end{align*}
where, $H^1_{\Gamma_D}(\Omega)^d := \left \{\vect{v} \in H^1(\Omega)^d : \vect{v}|_{\Gamma_D} = \vect{0} \right\}$, $|_{\Gamma_D}$ is understood in the boundary trace sense, $\vect{u} = \vect{u}(\rho)$ denotes the displacement of the structure, $\sigma$ denotes the stress tensor, the traction $\vect{f} \in H^{1/2}(\Gamma_N)^d$ is given, $\Gamma_N , \Gamma_D \subset \partial \Omega$ are known boundaries on $\partial \Omega$ such that $\Gamma_N \cup \Gamma_D = \partial \Omega$, $\mu_l$ and $\lambda_l$ are the Lam\'e coefficients, $\text{tr}(\cdot)$ is the matrix-trace operator, $\mathbb{I}$ is the $d \times d$ identity matrix, $\vect{n}$ is the outward normal and 
\begin{align*}
\varepsilon(\vect{u}) = \frac{1}{2}( \nabla \vect{u} + \nabla \vect{u}^\top), \;\;\;
k(\rho) = \epsilon_{\text{SIMP}}+ (1-\epsilon_{\text{SIMP}}) \rho^{p_s},
\end{align*} 
where $0< \epsilon_{\text{SIMP}} \ll1$ and $p_s \geq 1$. Unless stated otherwise, we choose $\epsilon_{\text{SIMP}} = 10^{-5}$ and $p_s=3$. The use of $k(\rho)$ is known as the Solid Isotropic Material with Penalization (SIMP) model. Bends\o e and Sigmund \cite[Ch.\ 1]{Bendsoe2004} provide a concise physical interpretation of the SIMP model. In essence, for $\rho$ close to one, $k(\rho)$ is close to one, indicating the presence of material, whereas where $\rho$ is close to zero, $k(\rho)$ approaches $\epsilon_{\text{SIMP}}$, indicating void. Thus, $k$ is the reverse of the inverse permeability, $\alpha$. It is typical to raise $\rho$ to the power of $p_s>1$ in order to penalize intermediate values of $\rho$. 

We introduce a Lagrange multiplier $\vect{v} \in H^1_{\Gamma_D}(\Omega)^d$ and reformulate \cref{complianceopt} as finding the stationary points $(\vect{u}, \rho, \vect{v})$ of
\begin{align}
\begin{split}
\int_{\Gamma_N} \vect{f} \cdot \vect{u}\; \text{d}s 
+ \int_\Omega k(\rho) \left[ 2 \mu_l \varepsilon(\vect{u}) : \varepsilon(\vect{v})+ \lambda_l \text{tr}(\varepsilon(\vect{u})) \cdot \text{tr}(\varepsilon(\vect{v}))  \right] \text{d}x  
- \int_{\Gamma_N} \vect{f}\cdot \vect{v} \; \text{d}s
\end{split} \label{eq:compliance1}
\end{align}
such that $ 0 \leq \rho \leq 1$ a.e.\ in $\Omega$,  and
$\int_\Omega \rho \; \text{d}x = \gamma |\Omega|$.

By deriving the Euler--Lagrange equations of \cref{eq:compliance1}, we see that the linear elasticity PDE constraint on $\vect{u}$ must be satisfied. However, if we consider the adjoint equation involving $\vect{v}$, it can be verified that $\vect{v} = -\vect{u}$. Substituting this relation into \cref{eq:compliance1}, we see that \cref{eq:compliance1} is equivalent to finding the stationary points of
\begin{align}
\begin{split}
2\int_{\Gamma_N} \vect{f} \cdot \vect{u}\; \text{d}s 
- \int_\Omega k(\rho) \left[ 2 \mu_l \varepsilon(\vect{u}) : \varepsilon(\vect{u})+ \lambda_l \text{tr}(\varepsilon(\vect{u})) \cdot \text{tr}(\varepsilon(\vect{u}))  \right] \text{d}x 
\end{split} \label{eq:compliance2}
\end{align}
such that $0 \leq \rho \leq 1$ a.e.\ in $\Omega$, and
$\int_\Omega \rho \; \text{d}x = \gamma |\Omega|$.
The substitution is useful as it greatly reduces the size of the problem after discretization. 

Unfortunately, the problem in general is ill-posed and does not have minimizers in the continuous setting. Na\"ive attempts at finding minimizers often yield checkerboard patterns of $\rho$. Although a different choice of finite element spaces may avoid the checkerboarding, the solutions will still be mesh-dependent. As the mesh is refined, the beams of the solutions will become ever thinner, leading to nonphysical solutions in the limit. There are several schemes employed by the topology optimization community to obtain physically reasonable solutions for $\rho$ and they are known as \textit{restriction methods} \cite{Bendsoe2004}. We opt for the addition of a Ginzburg--Landau energy term,
 \begin{align*}
J_{\mathrm{GL}}(\vect{u},\rho) := J(\vect{u},\rho) + \frac{\beta \epsilon}{2}\int_\Omega |\nabla \rho|^2 \; \text{d}x + \frac{\beta}{2 \epsilon} \int_\Omega \rho(1-\rho) \text{d}x,
 \end{align*}
with $0 < \beta \ll1$, $0 <\epsilon \ll 1$, to the objective function. $J_{\mathrm{GL}}$ requires $\rho$ to be weakly differentiable. Hence we now seek a solution $\rho \in C_\gamma \cap H^1(\Omega)$. Physically, the Ginzburg--Landau term corresponds to penalizing  fluctuations in the values of $\rho$. As $\epsilon \to 0$, it was shown by Modica \cite{Modica1987} that the Ginzburg--Landau energy $\Gamma$-converges to the perimeter functional associated with restricting $\rho(x) \in \{0,1\}$, providing rigorous mathematical grounding for this choice of regularization. For sufficiently large values of $\beta$, this introduces minima and removes the checkerboarding effect. Other restriction methods used by the topology optimization community include gradient control \cite{Borrvall2001a}, perimeter constraints \cite{Borrvall2001a}, sensitivity filtering \cite{Bourdin2001, Sigmund1994}, design filtering \cite{Bruns2001, Lazarov2011} and regularized penalty \cite{Borrvall2001a}. 

After these manipulations, the Lagrangian is given by
\begin{align*}
L(\vect{u},\rho, \lambda) &:= 2\int_{\Gamma_N} \vect{f} \cdot \vect{u}\; \text{d}s 
- \int_\Omega k(\rho) \left[ 2 \mu_l \varepsilon(\vect{u}) : \varepsilon(\vect{u})+ \lambda_l \text{tr}(\varepsilon(\vect{u})) \cdot \text{tr}(\varepsilon(\vect{u}))  \right] \text{d}x \\
& \indent + \frac{\beta \epsilon}{2} \int_\Omega |\nabla \rho|^2 \text{d}x + \frac{\beta}{2 \epsilon} \int_\Omega \rho(1-\rho) \text{d}x -  \int_\Omega \lambda (\gamma - \rho) \text{d}x,
\end{align*}
where $\lambda \in \mathbb{R}$ is the Lagrange multiplier for the equality volume constraint. We then define the enlarged feasible-set barrier functional as in  \cref{eq:enlargedbarrier}.

We have formulated enlarged feasible-set barrier functionals for both Borrvall--Petersson and structural compliance optimization problems. Finding stationary points of these barrier functionals is equivalent to computing minima, maxima and saddle points of the underlying optimization problems. In the next section we will introduce our algorithm and explain how we obtain multiple stationary points.

\section{The deflated barrier method}
\label{sec:alg}
In the following sections, we describe the components of the deflated barrier method. More specifically, we justify the usage of a barrier method where the subproblems are solved with a primal-dual active set solver to handle the effects of the barrier parameter in the Hessian. This is in contrast to a direct application of a discretize-then-optimize (DTO) primal-dual interior method, which does not use the structure of the original infinite-dimensional optimization problem. In the context of PDE-constrained optimization, ignoring the problem structure often results in mesh-dependence of the solver. Mesh-dependence is the phenomenon whereby with each refinement of the mesh, the number of iterations required by the optimization algorithm increases in an unbounded way \cite{Schwedes2017}.
\subsection{Choosing a solver for the subproblems}
Approximately solving the first order conditions of ${L^0_\mu}$ as $\mu \to 0$ is the classical primal interior point approach to finding the minima of \cref{borrvallmin} and \cref{complianceopt}. Without additional care, a direct implementation results in the following poor numerical behavior: 
\begin{enumerate}[label=({B}\arabic*)]
\setlength\itemsep{0em}
\item The Hessian of $L^0_{\mu_k}(\vect{z})$ has condition number $\mathcal{O}(1/\mu_k)$. Hence as $\mu$ decreases, the computed Newton updates may become inaccurate and require more solver time \cite[Th.\ 4.2]{Forsgren2002};
\label{illcondition} 
\item An initial guess of $\vect{z}_* = \vect{z}_{k}$ for the subproblem $\mu = \mu_{k+1}$ is asymptotically infeasible if an exact full Newton update of the primal interior point method is used. More precisely, if $\delta\rho^0_{k+1}$ is the calculated Newton update for $\rho$ at the first iteration of the Newton solver at $\mu = \mu_{k+1}$, then as $\mu \to 0$, we see that $0 \leq \rho_{k} + \delta \rho^0_{k+1} \leq 1$ a.e.\ does not hold \cite[Sec.\ 4.3.3]{Forsgren2002}.
\label{infeasiblenewton}
\end{enumerate}
\vspace{2mm}
Typically, to avoid the poor numerical behavior of \labelcref{illcondition} and \labelcref{infeasiblenewton}, the DTO primal interior point method is reformulated as a primal-dual interior point method, eliminating the rational expressions. Since the problem is first discretized, the slack variables associated with box constraints are associated to the primal variable component-wise. This manifests as a block identity matrix within the full Hessian. The Hessian can then be reduced and the primal-dual approach is reformulated into a condensed form. 

It is well known that PDE-constrained optimization solvers suffer from mesh-dependence when they do not properly treat the structure of the underlying infinite-dimensional problem \cite{Schwedes2017}. In order to obtain accurate solutions, where it is clear if the material distribution indicates material or void, we may require several refinements of the mesh; in this context, it is clear that mesh-dependence would be particularly disadvantageous. The mesh-independence of our algorithm will be carefully studied in the subsequent numerical examples, and analyzed in future work.

In order to properly treat the structure of the underlying infinite-dimensional problem, we opt for an optimize-then-discretize (OTD) method. The full Hessian arising from an OTD primal-dual interior point method is no longer easily reduced, since the block associated with the slack variables is now a mass matrix, rather than the identity. To avoid solving uncondensed large systems involving three times the number of degrees of freedom of a primal approach, the goal is to develop an OTD barrier method that avoids the poor numerical behavior of \labelcref{illcondition} and \labelcref{infeasiblenewton}. In a novel approach, we achieve this by solving the subproblems arising from the first order conditions of the enlarged feasible-set barrier functional $L^{\epsilon_\text{log}}_\mu$, while still enforcing the true box constraints, $0 \leq \rho \leq 1$ a.e., with a primal-dual active set solver. Whereas in a standard barrier method, the barrier terms act as a replacement for the box constraints on $\rho$, \textit{here we retain the box constraints to be handled by the primal-dual active set solver}. The barrier terms are instead used for continuation of the problem, to aid global convergence and to search for other branches of solutions.

The two inner solvers we consider are Hinterm\"uller et al.'s \textit{primal-dual active set strategy} (HIK) \cite{HintermullerIto2003} and Benson and Munson's \textit{active-set reduced space strategy} (BM) \cite{Benson2003}. We briefly illustrate the basic approach taken to solve the individual  subproblems using the log-barrier approach coupled with a primal-dual active set solver.  Let $J: \mathbb{R}^n \to \mathbb{R}$ be a twice-continuously differentiable function and consider the following box-constrained nonlinear program:
\begin{align}
\min_{z \in \mathbb{R}^n} J(z) \;\; \text{subject to} \;\; a \leq z \leq b. \label{boxconstrainedopt}
\end{align}
Here, we assume that $a, b \in \mathbb R^n$ such that $a < b$ (in each component) and we understand the inequality constraints $a \le z \le b$ component-wise.  Next, we formulate an `outer approximation' of \cref{boxconstrainedopt} using enlarged feasible-set log-barrier terms (for any $\mu, \epsilon_{\text{log}} > 0$):
\begin{align*}
\min_{z \in \mathbb{R}^n}
\left\{J(z) - \mu  \sum_{i=1}^n[\log(z_i - (a_i-\epsilon_{\text{log}}))+ \log((b_i +\epsilon_{\text{log}})- z_i)] : a \leq z \leq b \right\}  .
\end{align*}
We emphasize that there are two pairs of box constraints: the true box constraints $[a,b]$ and the enlarged feasible-set box constraints $[a-\epsilon_{\text{log}}, b+\epsilon_{\text{log}}]$, $\epsilon_{\text{log}}>0$. For any fixed $\mu>0$, the associated KKT-system has the form
\begin{align}
F(z) - \lambda^a + \lambda^b &= 0, \label{eq:kkt4}\\
\lambda^a, \; \lambda^b&\geq 0,\label{eq:kkt5}\\
z-a \geq 0, \; b-z&\geq 0,\label{eq:kkt6} \\
\langle \lambda^a, z-a \rangle_{(\mathbb{R}^n)^*, \mathbb{R}^n} = \langle \lambda^b, b-z\rangle_{(\mathbb{R}^n)^*, \mathbb{R}^n}  & = 0, \label{eq:kkt7}
\end{align}
where, $\lambda^a, \lambda^b \in (\mathbb{R}^n)^*$ are Lagrange multipliers associated with the true box constraints and
\begin{align}
F(z):=J'(z) -\frac{\mu}{z - (a-\epsilon_{\text{log}})}+ \frac{\mu}{b+\epsilon_{\text{log}} - z},
\end{align}
where the rational expressions are interpreted component-wise. The equivalent mixed complementarity problem is given by 
\begin{alignat}{2}
\text{either} &\;\; a_i < z_i < b_i \;\; &&\text{and} \;\; F(z)_i = 0, \label{eq:mcp1}\\
\text{or} &\;\; a_i = z_i \;\; &&\text{and} \;\; F(z)_i \geq 0,\\
\text{or} &\;\; z_i = b_i \;\; &&\text{and} \;\; F(z)_i \leq 0. \label{eq:mcp3}
\end{alignat}
Consider the natural residual function $\varphi(x,y) = x - (x-y)_+$ where $(\cdot)_+:=\max(\cdot, 0)$. This is an example of an NCP function, a class of functions that for $x, y \in \mathbb{R}$ satisfy
\begin{align}
\varphi(x,y) = 0 \;\; \text{if and only if} \;\; x,y \geq 0, \;\; xy = 0.
\end{align}
Using $\varphi$, we note that \cref{eq:kkt4}--\cref{eq:kkt7} can be reformulated as the following:
\begin{align}
F(z) - \lambda^a + \lambda^b &= 0, \label{eq:kkt1}\\
\varphi(\lambda^a, z-a) = \lambda^a - (\lambda^a - (z - a))_+&= 0,\label{eq:kkt2}\\
\varphi(\lambda^b, b-z) = \lambda^b - (\lambda^b - (b - z))_+&= 0.\label{eq:kkt3}
\end{align}
Assuming we are given a strictly enlarged-set feasible iterate $z \in \mathbb{R}^n$, $a - \epsilon_{\text{log}} < z < b+\epsilon_{\text{log}}$, we linearize around the point $(z,\lambda^a,\lambda^b)$ using the associated Newton-derivative and reduce the system based on the estimates of the active and inactive sets predicted by the semismooth Newton step. 

In HIK, the linearized system in the direction of $(\delta z,\delta\lambda^a,\delta\lambda^b)$ is given by 
\begin{align}
\begin{split}
F'(z) \delta z - \delta \lambda^a + \delta \lambda^b 
=-F(z) + \lambda^a - \lambda^b,
\label{eq:HIKstep}
\end{split}
\end{align}
where $F'(z) \in \mathbb{R}^{n \times n}$ denotes the Fr\'echet derivative of $F$ and 
\begin{align}
z_i + \delta z_i &= a_i \quad &&\text{ if } i \in \mathfrak{A}^a =\{i:\lambda^a_i - z_i + a_i > 0\}, \label{eq:activelower}\\
z_i + \delta z_i &= b_i \quad &&\text{ if }  i \in \mathfrak{A}^b =\{i:\lambda^b_i - b_i + z_i > 0\} \label{eq:activehigher},\\
\lambda^a_i + \delta \lambda^a_i &= 0 \quad &&\text{ if }  i \in \mathfrak{I}^a =\{i:\lambda^a_i - z_i + a_i \leq 0\},\\
\lambda^b_i + \delta \lambda^b_i &= 0 \quad &&\text{ if }  i \in \mathfrak{I}^b =\{i:\lambda^b_i - b_i + z_i \leq 0\}. \label{eq:inactivehigher}
\end{align}
We define the active set by  $\mathfrak{A} = \mathfrak{A}^a \cup \mathfrak{A}^b$ and the inactive set by $\mathfrak{I} = \mathfrak{I}^a \cap \mathfrak{I}^b$. By substituting \cref{eq:activelower}--\cref{eq:inactivehigher} into \cref{eq:HIKstep} and removing the rows associated with the active set, we observe that
\begin{align}
F'(z)_{\mathfrak{I},\mathfrak{I}}\delta z_\mathfrak{I}
=- F'(z)_{\mathfrak{I}, \mathfrak{A}} \delta z_{\mathfrak{A}}-F(z)_{\mathfrak{I}}. \label{eq:reducedHIK}
\end{align}
We can therefore solve the reduced linear system \cref{eq:reducedHIK} to find the remaining unknown components of $\delta z$. 

BM attempts to solve \cref{eq:mcp1}--\cref{eq:mcp3} as follows. Given a feasible iterate $z$ with respect to the true box constraints, $a \leq z \leq b$, the active set is defined by
\begin{align}
\mathcal{A} = \{i: z_i=a_i \; \text{and} \; F(z)_i > 0 \} \cup \{i: z_i=b_i \; \text{and} \; F(z)_i <  0 \}, 
\end{align}
and the inactive set is given by $\mathcal{I} = \{i\}_{i=1}^n\backslash \mathcal{A}$. The linearized system in the direction of $\delta z$ takes the form
\begin{align}
F'(z)_{\mathcal{I},\mathcal{I}}\delta z_\mathcal{I} = -F(z)_{\mathcal{I}} \;\; \text{and} \;\; \delta z_{\mathcal{A}} = 0. \label{eq:reducedBM}
\end{align}
The next iterate is then given by $\pi(z+\delta z)$, where $\pi$ is the component-wise projection onto the true box constraints, i.e.\ 
\begin{align}
\pi(z+\delta z)_i = 
\begin{cases}
a_i & \text{if} \; z_i+\delta z_i < a_i,\\
 z_i+\delta z_i&\text{if} \; a_i \leq z_i+\delta z_i \leq b_i,\\
 b_i & \text{if} \; z_i+\delta z_i > b_i.
 \end{cases}
\end{align}
The HIK solver is a well-established method and under suitable assumptions is equivalent to a semismooth Newton method \cite{Qi1993, Qi1993b, Ulbrich2003} in both finite and infinite-dimensions \cite{HintermullerIto2003}. This equivalence ensures local superlinear convergence and under further assumptions guarantees mesh-independence \cite{Hintermuller2004}. Until now, the BM solver had no supporting theoretical results, although is conveniently included in PETSc \cite{petsc}. Experimentally, we observe that the BM solver enjoys superlinear convergence. At first glance, the two solvers may appear quite different, but in \cref{sec:BMsolver} we prove that for a linear elliptic control problem, if the active and inactive sets coincide between the two algorithms, then the updates given by HIK and BM are identical.

One common critique of barrier methods is that the step size rules for the update of the distributed control go to zero. We observe this in numerical examples if we use a Newton solver; however, this issue is averted when using HIK or BM. A step size of one is always taken for the update of the primal variable's active set, whereas a linesearch can be used for the update of the primal variable's inactive set. Hence, areas of the domain where the control attains the constraint do not influence the step sizes of the updates for sections of the control which are strictly feasible.

Both HIK and BM perform a pointwise projection on the iterates generated by the subproblems of the barrier functional. In the context of a classical OTD primal-dual interior point method applied to a PDE-constrained optimal control problem, under certain assumptions, Ulbrich and Ulbrich \cite{Ulbrich2000, Ulbrich2009} prove that local superlinear convergence holds if the iterates of the control and its associated Lagrange multipliers are pointwise projected to a controlled neighborhood of the central path. Although not all their assumptions hold in our case (in particular these problems are not convex), the combination of a primal-dual active set solver and barrier method mimics the computation of a Newton step of a primal-dual approach and then performing a pointwise projection. An advantage of our method is that our pointwise projection is unique and cheap to compute.
 
Numerically, this method only requires solving linear systems that are less than or equal to the size of the linear systems in a standard barrier method. Moreover, in the BM solver, the constrained variables can never reach the bounds of the enlarged feasible-set, ensuring the Hessian remains bounded. Furthermore, both the BM and HIK solvers remove the rows and columns in the Hessian associated with the active constraints. It is these active constraints which are the source of the unbounded eigenvalues that cause the ill-conditioning of the barrier method as $\mu$ approaches zero.  In \cref{fig:conditionnumber} we give an example demonstrating that the condition number is controlled by the elimination of the active set. Removing rows and columns associated with the active set mimics the principle of Nash et al.'s \textit{stabilized barrier method} \cite{Nash1994, Nash1993}.   

\subsection{Deflation}
Deflation is an algorithm for the calculation of \textit{multiple} solutions of systems of nonlinear equations from the same initial guess. Let $V$ and $W$ be Banach spaces. Suppose a system of PDEs, $F(z) = 0$, $F: V \to W$ has multiple solutions $z = z_1, \dots z_n$, that we wish to find. We find the first solution by utilizing a Newton-like algorithm to find $z_1$. Now instead of using a standard multistart approach which may converge to the same solution, we instead introduce a modified system $G(z) = 0$ such that:
\begin{enumerate}
\setlength\itemsep{0em}
\item $G(z) = 0$ if and only if $F(z) = 0$ for $z \neq z_1$;
\item A Newton-like solver starting from any initial guess $z_* \neq z_1$ applied to $G$ will not converge to $z_1$. 
\end{enumerate}
\begin{figure}[ht]
\centering
\subfloat[Before deflation.]{\includegraphics[width = 0.4\textwidth]{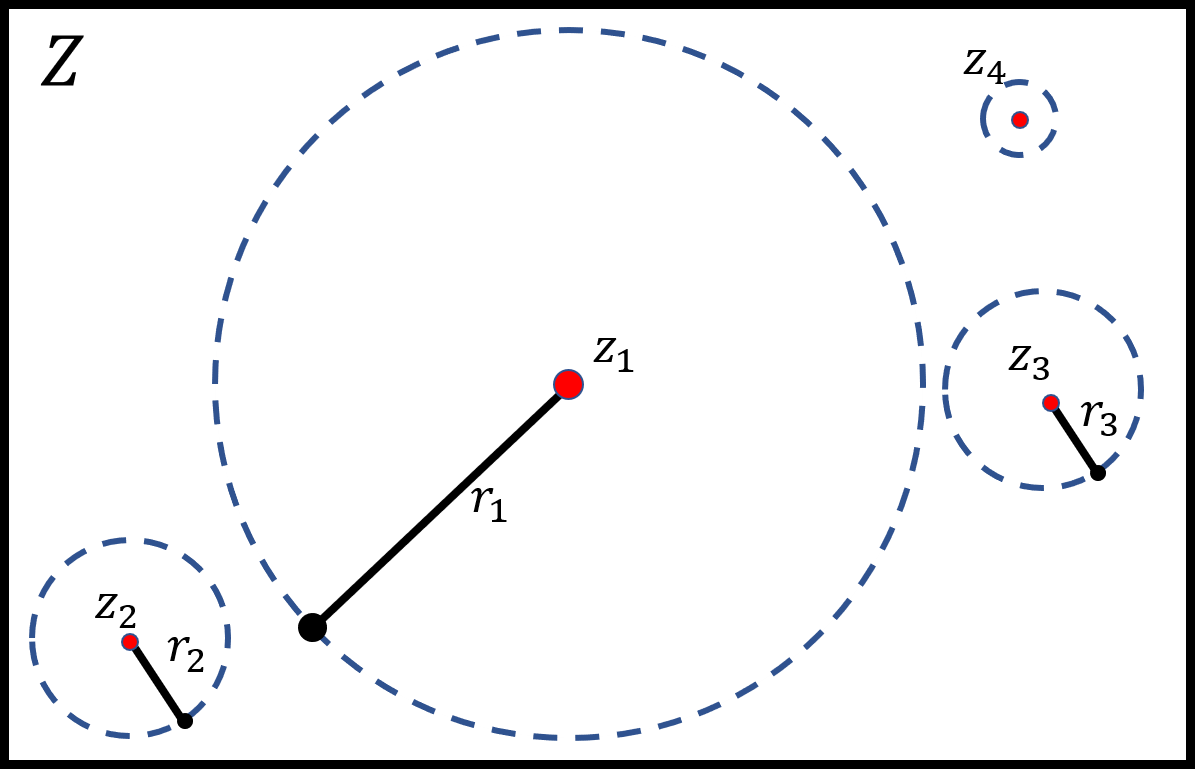}} \;\;
\subfloat[After the deflation of $z_1$.]{\includegraphics[width = 0.4\textwidth]{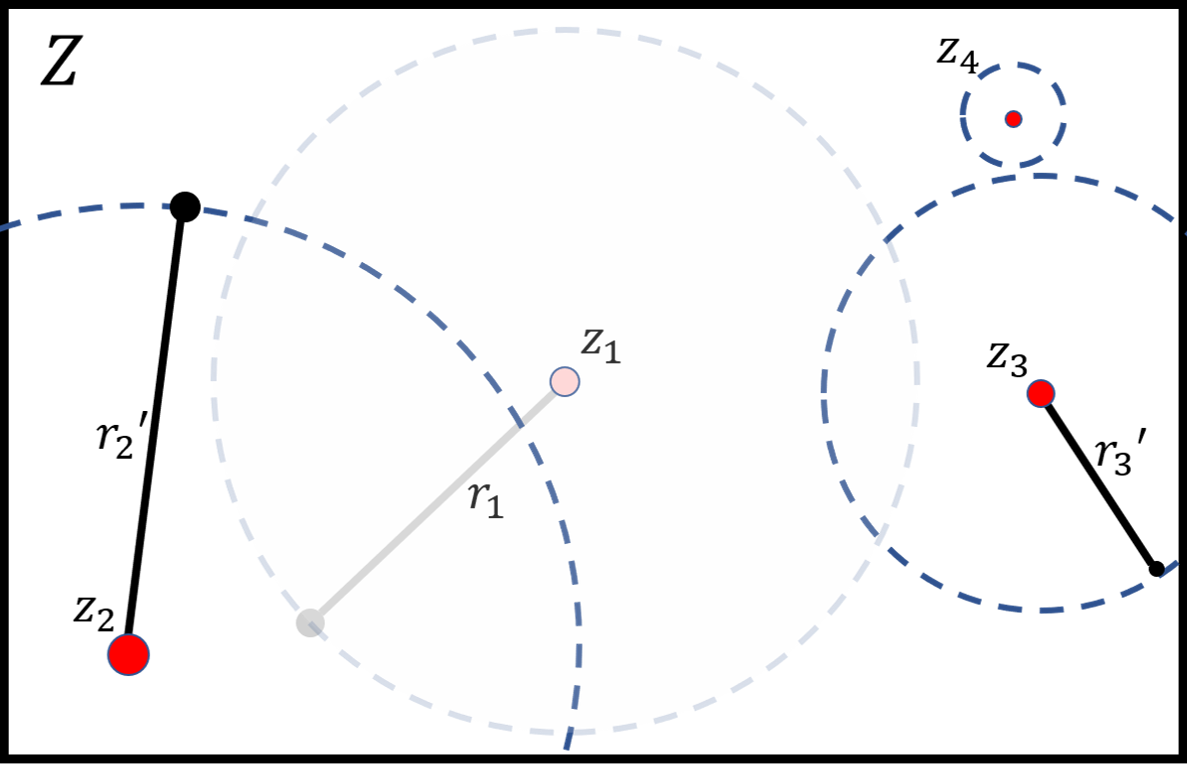}}
\caption{The solutions $z_1, z_2, z_3$ and, $z_4$ are zeros of the system $F(z)$. The circles around the solutions represent the basins of attraction within which a Newton-like solver converges to that particular solution.}\label{fig:deflation}
\end{figure}
This process is visualized in \cref{fig:deflation}. In principle, one can use the same initial guess to converge to multiple solutions. The modified system is obtained by applying a \textit{deflation operator}, $\mathcal{M}(z; z_1):W \to W$ to $F$ such that:
\begin{enumerate}[label = {(D\arabic*)}]
\item $\mathcal{M}(z; z_1)$ is invertible for all $z\neq z_1$ in a neighborhood of $z_1$; \label{deflation1}
\item $\liminf_{z\to z_1}\| \mathcal{M}(z; z_1) F(z) \| > 0$. \label{deflation2}
\end{enumerate}
\labelcref{deflation1} ensures that the resulting system has a solution if the original problem has an unknown solution, and \labelcref{deflation2} ensures that a Newton-like method applied to the newly deflated system does not converge as  $z \to z_1$. In this work we consider the shifted deflation operator $\mathcal{M}(z;z_1) = (\| z - z_1 \|^{-2}_{V} + 1)\mathcal{I}$, where $\mathcal{I} : W \to W$ is the identity operator \cite{Farrell2015}. In particular, in all the numerical examples discussed in \cref{sec:numerical}, deflation is implemented with respect to the material distribution, i.e.~$\mathcal{M}(\vect{z};\vect{z}_1) = (\| \rho - \rho_1 \|^{-2}_{L^2(\Omega)} + 1)\mathcal{I}$, where $\vect{z} = (\vect{u}, \rho, p, p_0, \lambda)$ and $\vect{z} = (\vect{u}, \rho, \lambda)$ in fluid and compliance problems, respectively. 

Deflation can be implemented very efficiently. In particular, the conditioning of the Jacobian of the deflated system does not cause computational difficulty, since the Newton update of the discrete deflated system is expressed as a scaling of the Newton update of the original discrete undeflated system via the Sherman--Morrison formula \cite[Sec.\ 3]{Farrell2015}. Let $F_h: V_h \to W_h$ be an approximation to $F$ on the finite-dimensional spaces $V_h$ and $W_h$. Let $\delta z_h$ denote the solution of the deflated Newton system evaluated at $z_h \in V_h$, to be computed, and let $\delta y_h$ denote the solution of the \emph{undeflated} Newton system of $F_h$, assembled at the same current iterate $z_h$. Let $\vectt{z}$, $\delta \vectt{z}$, and $\delta \vectt{y}$ be the discrete coefficient vectors of $z_h$, $\delta z_h$, and $\delta y_h$, respectively. Moreover, let $m(\vectt{z}) = \mathcal{M}(z_h, z_{1,h})$ and denote the derivative of $m$ with respect to $\vectt{z}$ by $m'(\vectt{z})$. The solution of the discrete deflated Newton system can be computed by scaling $\delta \vectt{y}$ \cite[Sec.\ 3]{Farrell2015}:
\begin{align} 
\delta \vectt{z}  = \left( 1  + \frac{m^{-1} (m')^\top (\delta \vectt{y})}{1 - m^{-1} (m')^\top (\delta \vectt{y})} \right) \delta \vectt{y}. \label{eq:deflationscaling}
\end{align}
The same formula applies if multiple solutions have been deflated, i.e.~if $m(\vectt{z}) = \mathcal{M}(z_h, z_{1,h}) \cdots \mathcal{M}(z_h, z_{n,h})$ for $n > 1$. The simple structure of \eqref{eq:deflationscaling}
arises because the deflated residual is a (nonlinear) scalar multiple of the original residual.

In summary, in order to compute the update $\delta \vectt{z}$ for the discretized deflated system, only the original, discretized, undeflated system is solved. Its solution $\delta \vectt{y}$ is then scaled as in \cref{eq:deflationscaling}.

Deflation was first introduced in the context of polynomials by Wilkinson \cite{Wilkinson1963}. It was then extended to differentiable finite-dimensional maps $F:\mathbb{R}^n \to \mathbb{R}^n$ by Brown and Gearhart \cite{Brown1971}. More recently, Farrell et al.\ extended the original Brown and Gearhart technique to Fr\'echet-differentiable maps between Banach spaces \cite{Farrell2015}. Deflation has been used to discover multiple solutions of cholesteric liquid crystals, Bose--Einstein condensates, mechanical metamaterials, aircraft stiffeners, and other applications \cite{charalampidis2016,Emerson2018,medina2019a,Robinson2017,xia2019a}. It has also been extended to semismooth mappings \cite{Farrell2019}, which is necessary in the current context of topology optimization. 
 
\subsection{Implementation of the deflated barrier method} \label{sec:dab}
The essential idea is to use deflation to attempt to find other branches during the continuation of the barrier parameter, as visualized in \cref{fig:visualization}. As summarized in \cref{fig:flowchart}, the deflated barrier method is divided into three phases: prediction, continuation and deflation. 
\newline \textbf{Prediction:} Given a solution $z_{k-1}$ at $\mu = \mu_{k-1}$, the algorithm calculates an initial guess for the corresponding solution at $\mu = \mu_k < \mu_{k-1}$. This is done via a feasible tangent prediction method (as described in \cref{sec:tangentpredictiontask}), a classical tangent prediction method \cite[Sec.\ 4.4.1]{Seydel2010} or a secant prediction method \cite[Sec.\ 4.4.2]{Seydel2010}. A feasible tangent prediction method is identical to its classical counterpart but with box constraints on the predictor step to ensure the initial guess is feasible.  
\newline \textbf{Continuation:} Given an initial guess for each branch at the new barrier parameter $\mu_k$, the algorithm calculates the new solution along each branch with a primal-dual active set solver whilst deflating away all solutions already known at $\mu = \mu_k$.
\newline \textbf{Deflation:} At some subset of the continuation steps, the algorithm searches for new branches at $\mu = \mu_k$ using solutions on different branches found at $\mu  = \mu_{k-1}$ as initial guesses. The search terminates when all the initial guesses have been exhausted (reached a maximum number of iterations without converging) or when a certain number of branches $\beta_{\mathrm{max}}$ have been found.

We now explain the notation used in \cref{alg:dbm}. Let $\vect{z}=(\vect{u},\rho,p,p_0,\lambda)$ in the Borrvall--Petersson case and $\vect{z}= (\vect{u},\rho,\lambda)$ in the compliance case. The value of the barrier parameter at subproblem iteration $k$ is denoted $\mu_k$. The initial guess for the density is denoted $\rho_0$ and the initial guess for the volume constraint Lagrange multiplier is denoted $\lambda_0$. The generator for the next value of $\mu$ is denoted by $\Theta$. The $\mu$-update can be adaptive or chosen a priori, provided it gives a strictly decreasing sequence. Under suitable conditions, the first order conditions of $L^{\epsilon_{\text{log}}}_\mu(\vect{z})$ together with the box constraints on $\rho$ can be reformulated into perturbed KKT conditions \cite[Rem.\ 3]{Ulbrich2009} which in turn can be reformulated as a semismooth system of partial differential equations, $F_\mu(\vect{z})$. Let 
\begin{align}
\vect{y} = \begin{cases}
(\vect{u}, p, p_0) & \text{in the Borrvall--Petersson case},\\
\vect{u}              & \text{in the compliance case}.
\end{cases}
\end{align}
Let $'|_{\vect{z}_i}$ denote the Fr\'echet derivative with respect to $\vect{z}_i$.
Let $\mathcal{S}_{\mu_k}$ denote the set of solutions, $\{\vect{z}\}_i$, found at $\mu_k$. Let $\mathcal{M}(\cdot)$ denote the deflation operator and $Z$ denote the function space of $\vect{z}$.
\begin{algorithm}[ht]
\caption{Deflated barrier algorithm}
\label{alg:dbm}
\begin{algorithmic}[1]
\Initialize{
$k \gets 0$ \Comment{Initial iteration number}\\
$\mu_0$ \Comment{Initial barrier parameter}\\ 
tol \Comment{Approximate solve tolerance}\\
 $\beta_{\mathrm{max}}$ \Comment{Maximum number of branches sought}\\
$\rho_0(x) \gets \gamma$ \Comment{Constant initial material distribution} \\
$\lambda_0$ \Comment{Initial volume constraint multiplier}\\
\vspace{2mm}
}
\State{Approximately solve $(L^{\epsilon_{\text{log}}}_{\mu_0})'|_{\vect{y}}(\vect{y}, \rho_0) = 0$.}\Comment{Solve state equation for $\vect{y}$}
\State{$\vect{z}_* \gets (\vect{y}, \rho_0, \lambda_0)$} \Comment{Initial guess}
\State{Approximately solve $F_{\mu_0}(\vect{z}) = 0$ with initial guess $\vect{z}_*$.}
\State{$S_{\mu_0} \gets S_{\mu_0} \cup \{\vect{z}\}$} \Comment{Include solution in solution set }
\State{$\mu_1 \gets \Theta(\mu_0)$, $k \gets 1$} \Comment{Update $\mu$ and $k$}
\While{$\mu_k \ge 0$ and $|\mathcal{S}_{\mu_{k-1}} | \neq \varnothing$}
\For{$\vect{z}_i \in \mathcal{S}_{\mu_{k-1}}$}
\State{}\Comment{\textbf{Prediction}}
\State{Predict solution at $\mu_k$, denoted $\vect{z}_*$.}
\State{} \Comment{\textbf{Continuation}}
\State{Attempt to solve $\mathcal{M}\left(\mathcal{S}_{\mu_k}\right)  F_{\mu_k}(\vect{z}) = 0$ with initial guess $\vect{z}_*$.}
\If{$\|F_{\mu_k}(\vect{z}) \|_{Z^*} \leq \mathrm{tol}$}
\State{Solve has succeeded; set $S_{\mu_k} \gets S_{\mu_k} \cup \{\vect{z}\}$.}
\EndIf
\EndFor
\State{}\Comment{\textbf{Deflation}}
\For{$\vect{z}_j \in \mathcal{S}_{\mu_{k-1}}$}
\If{$|\mathcal{S}_{\mu_k}| \ge \beta_{\mathrm{max}}$}
\State{\textbf{break}}
\EndIf
\State{Attempt to solve $\mathcal{M}\left(\mathcal{S}_{\mu_k}\right) F_{\mu_k}(\vect{z}) = 0$ with initial guess $\vect{z}_j$.}
\If{$\|F_{\mu_k}(\vect{z}) \|_{Z^*} \leq \mathrm{tol}$}
\State{Solve has succeeded; set $S_{\mu_k} \gets S_{\mu_k} \cup \{\vect{z}\}$.}
\EndIf
\EndFor
\State{$\mu_{k+1} \gets  \Theta(\mu_k)$} \Comment{Choose new value of $\mu$}
\State{$k \gets k + 1$}
\EndWhile
\end{algorithmic}
\end{algorithm}
\begin{figure}[ht]
\centering
\includegraphics[width = 0.8\textwidth]{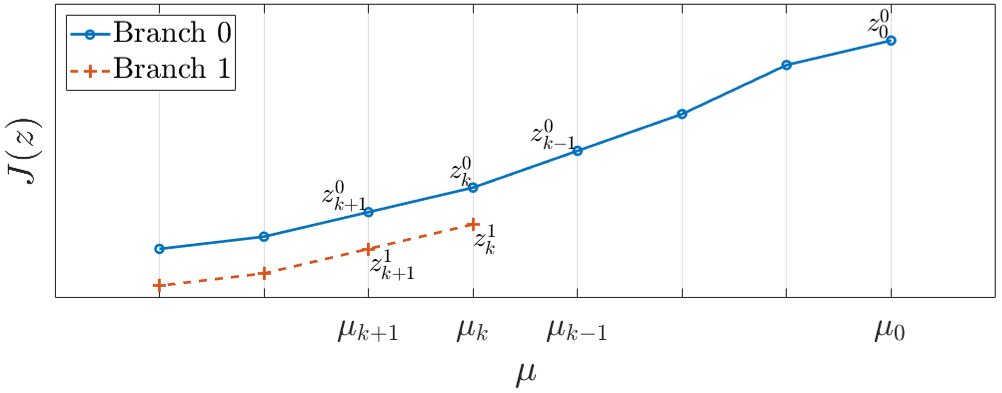}
\caption{A visualization of the deflated barrier method. Branch 0 is discovered at $\mu_0$. A predictor-corrector scheme is used to to follow the branch as $\mu$ decreases, denoted by circles. At $\mu = \mu_k$, deflation is used to discover a new solution on a different branch (branch 1), using the solution on branch 0 at $\mu = \mu_{k-1}$ as an initial guess. This newly discovered branch is then also continued as $\mu$ decreases, and is denoted by the crosses.} \label{fig:visualization}
\end{figure}
\begin{figure}[ht]
\centering
\includegraphics[width =0.8 \textwidth]{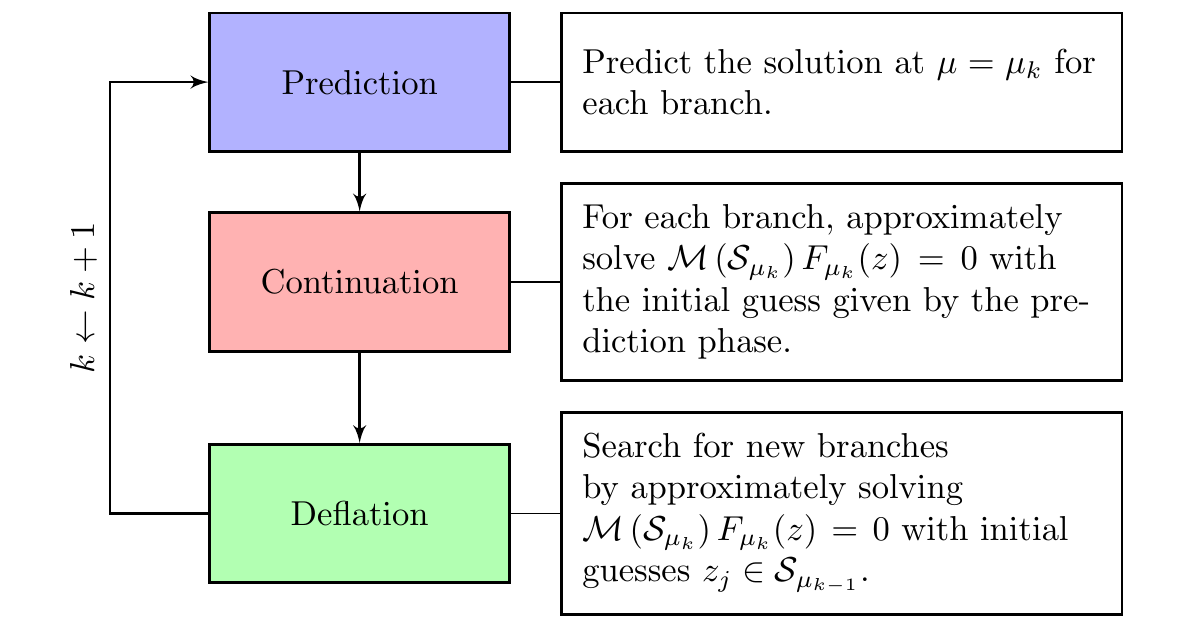}
\caption{A flowchart depicting the three phases involved in the deflated barrier method.} \label{fig:flowchart}
\end{figure}
\newpage
\section{Numerical results}
\label{sec:numerical}
In all examples the systems were discretized with the finite element method using FEniCS \cite{fenics} and the resulting linear systems were solved by a sparse LU factorization with MUMPS \cite{mumps} and PETSc \cite{petsc}. The meshes were either created in FEniCS or Gmsh \cite{Geuzaine2009}. We present three different examples of the minimization of the power dissipation of a fluid constrained by the Stokes equations, one constrained by the Navier--Stokes equations, and two examples of the minimization of the compliance constrained by linear elasticity. Throughout the numerical examples, $h_\text{min}$ denotes the minimum diameter of all simplices in the mesh, where the simplex diameter is defined as the maximum edge length. Similarly  $h_\text{max}$ denotes the maximum diameter of all simplices in the mesh. All solutions depicted are presented as computed by the deflated barrier method, with no truncation or postprocessing of the material distribution.

\subsection{Borrvall--Petersson double-pipe} 
\label{sec:sec:double-pipe}
We consider the double-pipe problem with volume fraction $\gamma = 1/3$, two prescribed flow inputs and two prescribed outputs, and the boundary conditions as prescribed in \cref{fig:doublepipe}. We use $\alpha$ as given in \cref{eq:alphachoice}, with $ \overline{\alpha} = 2.5 \times 10^4$ and $q = 1/10$. Here $q$ is a penalty parameter which controls the level of intermediate values (between zero or one) in the optimal design. 

We use a Taylor--Hood $(\mathrm{CG}_2)^2 \times \mathrm{CG}_1$ finite element discretization for the velocity and pressure and $\mathrm{CG}_1$ elements for the material distribution. For BM, we begin with $\mu_0 = 100$ and apply deflation immediately to find the second branch of solutions. For HIK, this strategy did not converge to the second branch, although the second branch is discovered with $\mu_0 = 105$. In both cases tangent prediction is used, as well as a damped $l^2$-minimizing linesearch \cite[Alg.\ 2]{Brune2015}.
\cref{fig:doublepipesolns} shows the minimizers of the double-pipe problem computed using the deflated barrier method.

\begin{figure}[ht]
\centering
\includegraphics[width = 0.6\textwidth]{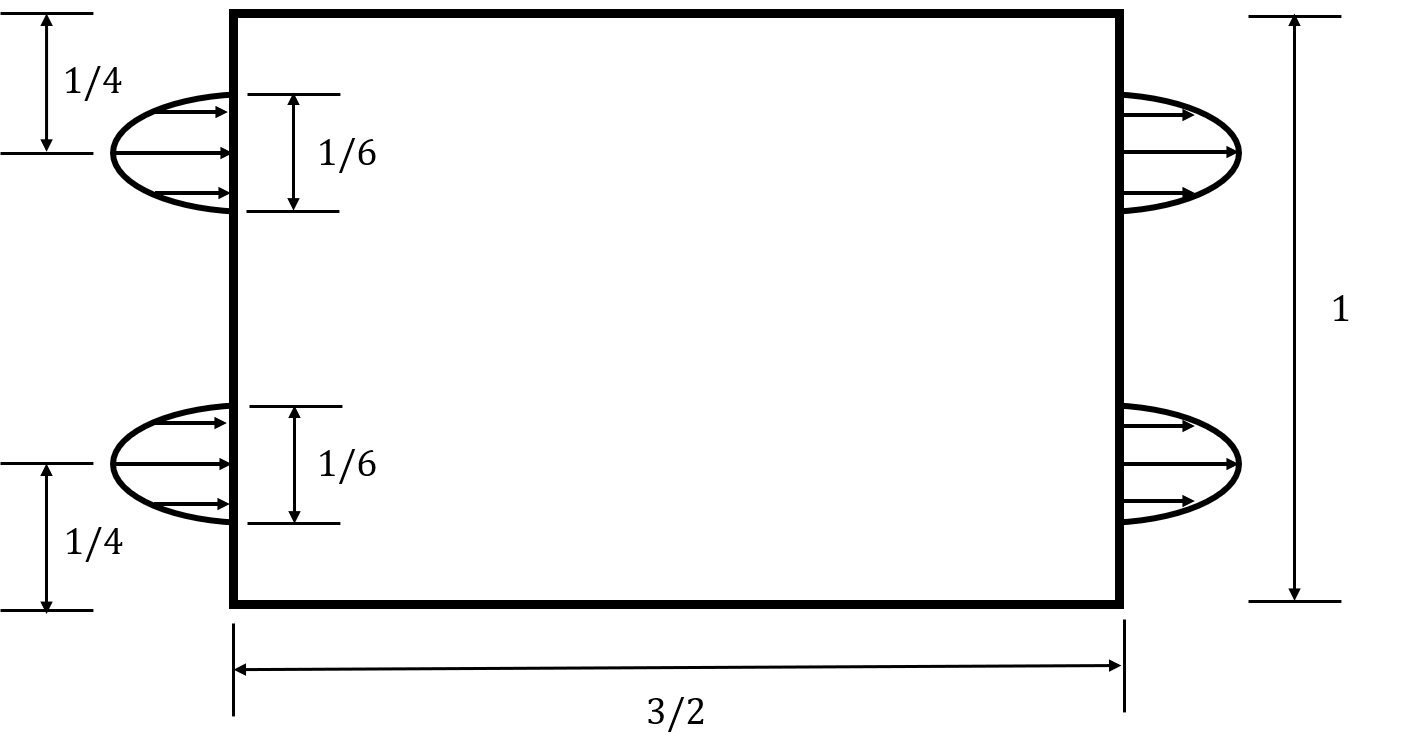}
\caption{Setup of the double-pipe problem. In our tests we pick $\vect{f} =(0,0)^\top$ and $\nu = 1$. The Dirichlet boundary conditions on the velocity are $\vect{u} = \left(1-144(y-3/4)^2, 0)\right)^\top$ for the top input and output boundary flows, $\vect{u} = \left(1-144(y-1/4)^2, 0)\right)^\top$ for the bottom input and output boundary flows and $\vect{u} = (0,0)^\top$ everywhere else.} 
\label{fig:doublepipe}
\end{figure}
\begin{figure}[ht]
\centering
\includegraphics[width = 0.49\textwidth]{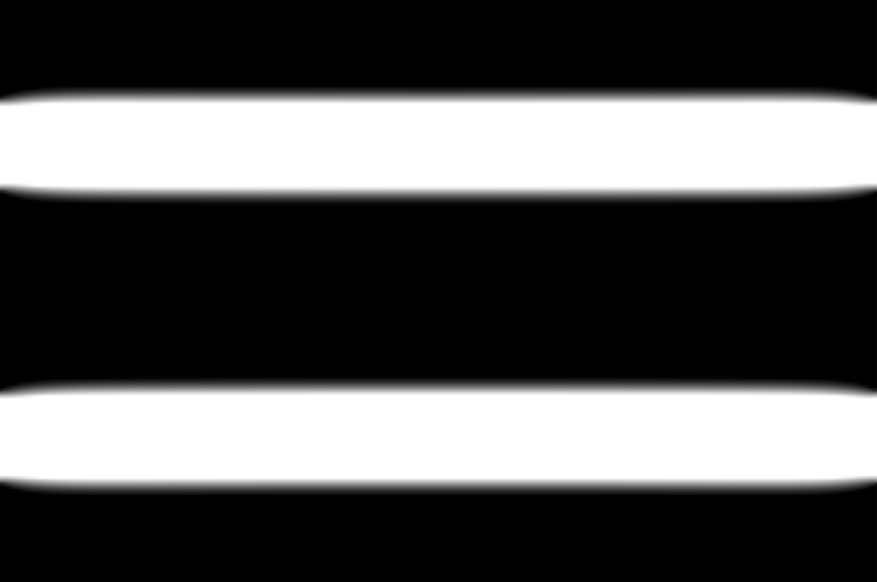}
\includegraphics[width = 0.49\textwidth]{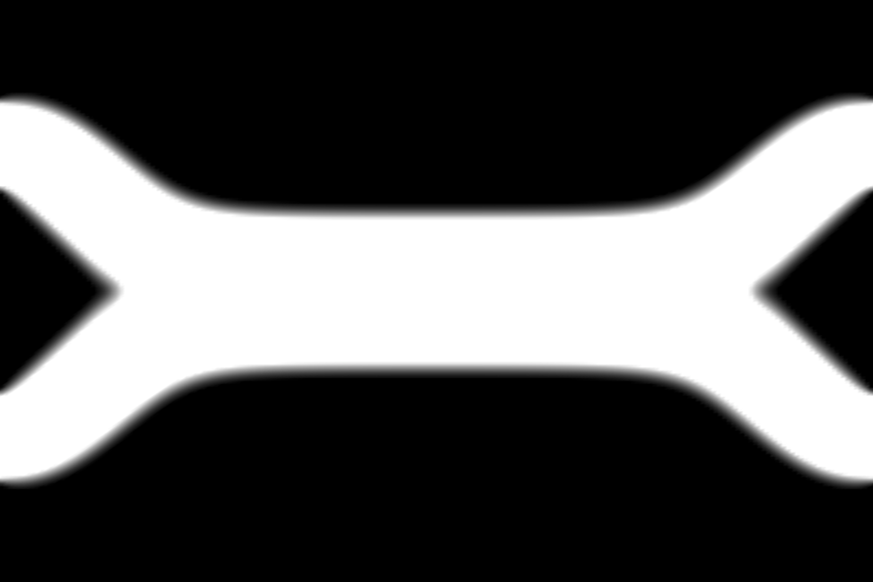}
\caption{The material distribution of the local (left) and global (right) minimizer of the double-pipe optimization problem with mesh size $h = 0.0141$. Black corresponds to a value of $\rho = 0$ and white corresponds to a value of $\rho = 1$. The objective functional values are $J = 32.58$ (left) and $J = 23.87$ (right).} 
\label{fig:doublepipesolns}
\end{figure}

In \cref{tab:doublepipeiters} we explore the mesh-independence of primal-dual active set solver iterations. 
We observe that with each refinement of the mesh, the number of iterations stay roughly constant. In particular, we notice that the behavior is consistent for both HIK and BM. This is a recurring theme and holds in subsequent examples.
To exemplify that the mesh-independence is not an artifact of our choice of finite element spaces, we also display the results of a divergence-free Scott--Vogelius $(\mathrm{CG}_2)^2 \times \mathrm{DG}_1$ finite element discretization for the velocity and pressure and $\mathrm{CG}_1$ for the material distribution. Stability of this discretization is ensured by using a barycentrically-refined mesh \cite{qin1994}.

In \cref{fig:conditionnumber} we plot the condition number of the Hessian as in a classical barrier method, and the condition number of the Hessian with the rows and columns associated with the active-set removed. We observe that the condition number of the latter is significantly smaller, accounting for why our proposed methodology does not suffer from ill-conditioning.
\begin{table}[ht]
\footnotesize
\centering
\begin{tabular}{ll|lll|lll}
BM Solver & Taylor--Hood& &Branch 0 &  &   & Branch 1 &  \\ \toprule
$h$&Dofs&Cont. & Defl. & Pred. & Cont. & Defl. & Pred. \\ \midrule
 0.0283& 38,256 &124 		& 0 		     & 22		  & 115 			 & 30  	 & 22   \\
 0.0177& 97,206 &	123 & 0		     &	22	  & 	109		&	30	 & 22 \\
 0.0141&  151,506&110		& 0		      &22		 &116			&29	 &22  \\ \bottomrule \vspace{1mm}
&&&&&&&\\
HIK solver & Taylor--Hood & & Branch 0 &  &   & Branch 1 &  \\ \toprule
$h$& Dofs &Cont. & Defl. & Pred. & Cont. & Defl. & Pred. \\ \midrule
 0.0283& 38,256 & 174 		& 0 		     & 43		  & 261			 & 14 	 & 43   \\
 0.0177& 97,206 & 189 		& 0 		     & 43		  & 223			 & 13 	 & 43   \\
 0.0141&  151,506&	173	& 		   0   &	43	 &		197	& 	13 & 43 \\ \bottomrule\vspace{1mm}
&&&&&&&\\
BM solver & Scott--Vogelius& & Branch 0 &  &   & Branch 1 &  \\ \toprule
$h_{\text{min}}$/$h_{\text{max}}$& Dofs &Cont. & Defl. & Pred. & Cont. & Defl. & Pred. \\ \midrule
 0.0278/0.0501& 58,685 & 155 		& 0 		     & 22		  & 139			 & 29 	 & 22   \\
 0.0139/0.0250& 234,005 & 124		& 0 		     & 22		  & 120			 & 29 	 & 22   \\
 \bottomrule
\end{tabular}
\caption{The cumulative total numbers of primal-dual active-set solver iterations required in the continuation, deflation and prediction phases of the double-pipe problem. Branch 0 discovers the local minimum shown in \cref{fig:doublepipesolns} and branch 1 discovers the global minimum. As we can see, the numbers of iterations stay roughly constant for both solvers as we refine the mesh.} 
\label{tab:doublepipeiters}
\end{table}

\begin{figure}[ht]
\centering
\includegraphics[width = 0.49\textwidth]{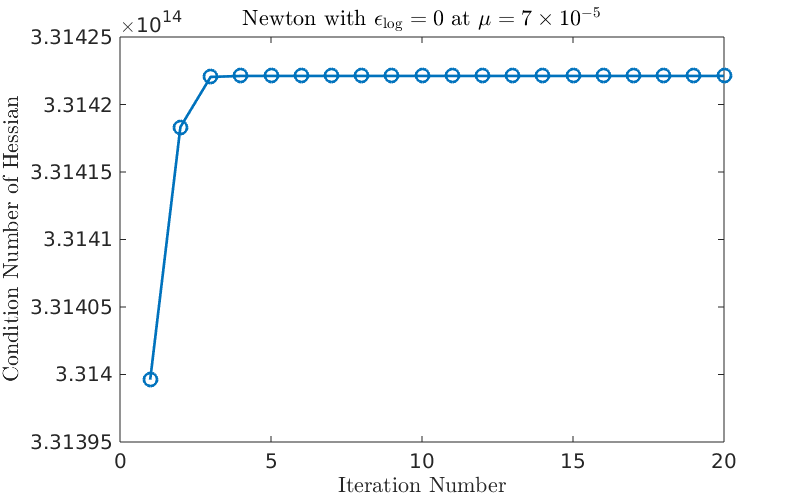}
\includegraphics[width = 0.49\textwidth]{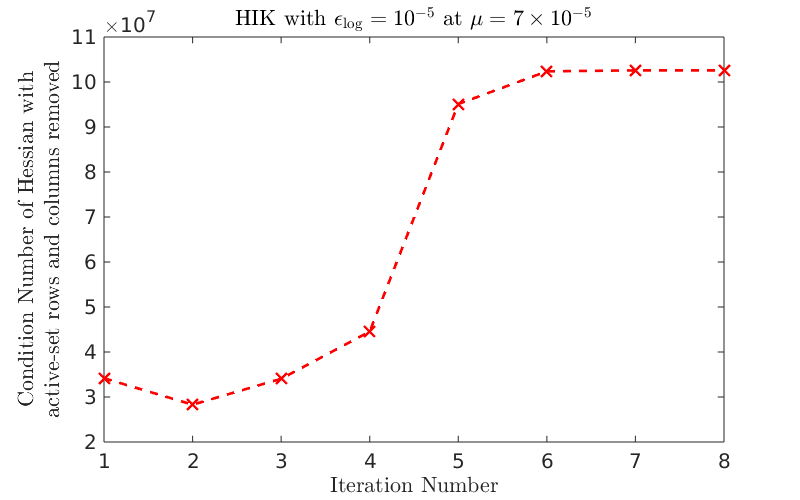}
\caption{The condition number of the Hessian at each iteration of the solver in the subproblem with $\mu = 7 \times 10^{-5}$. The condition number of the Hessian of $L^0_\mu$ arising in the linear systems of a standard Newton solver (left) is six to seven orders of magnitude larger than the condition number of the Hessian of $L^{\epsilon_{\text{log}}}_\mu$ arising in the linear systems of the HIK solver (right).} 
\label{fig:conditionnumber}
\end{figure}

\subsection{Neumann-outlet double-pipe} \label{sec:neumann-double-pipe}
One could argue that fixing the outlet flows is inherently nonphysical and a more realistic model would prescribe natural boundary conditions on the outlets (while keeping the Dirichlet boundary conditions on the inlets) \cite{Deng2018}. The correct choice of Neumann boundary conditions is nontrivial. Heywood et al.~\cite{Heywood1996} provide an investigation into various formulations. We opt for the natural boundary condition,
\begin{align}
\left(-p \mathbb{I} + 2\nu \varepsilon(\vect{u}) \right) \vect{n} = \vect{0} \; \text{on} \; \Gamma_N, \label{eq:neumannbcs}
\end{align}
where $\varepsilon(\vect{u}) := (\nabla \vect{u} + (\nabla \vect{u})^\top)/2$ denotes the symmetrized gradient, $\mathbb{I}$ denotes the $d \times d$ identity matrix and $\Gamma_N \subset \partial \Omega$ denotes the outlets. Heywood et al.~\cite{Heywood1996} note that such a formulation does not support Poiseuille flow. However, Limache et al.~\cite{Limache2007} proved that \cref{eq:neumannbcs} does satisfy the principle of objectivity, which is often violated by other common formulations, including $\left(-p \mathbb{I} + \nu \nabla \vect{u} \right) \vect{n} = \vect{0}$. The natural boundary condition \cref{eq:neumannbcs} is achieved by altering the objective functional to
\begin{align}
J_N(\vect{u},\rho) = \frac{1}{2} \int_\Omega \alpha(\rho) |\vect{u}|^2 + 2\nu |\varepsilon(\vect{u})|^2 \; \text{d}x. 
\label{eq:neumannfunctional}
\end{align}
Since $\text{div}((\nabla \vect{u})^\top) = \nabla (\text{div}(\vect{u}))$ and $\text{div}(\vect{u}) = 0$, we note that the minimizers of \cref{eq:neumannfunctional} are the same as those of the original functional, $J$, combined with the natural boundary conditions as described in \cref{eq:neumannbcs}. The other alteration in the optimization problem is the removal of the Lagrange multiplier, $p_0$, since the absolute pressure level is set by the outflow boundary condition.

We employ the Taylor--Hood discretization and initialize $\mu_0 = 1000$. Deflation finds the second, third and fourth branches at $\mu = 82.4$. For $h = 0.0333$, deflation discovers branch 2, then branch 1 and 3, whereas for the other mesh sizes, deflation discovers the branches in ascending order.

The removal of an imposed outlet flow has an interesting effect. The global minimizer in the shape of a double-ended wrench is now a local minimizer. Two new $\mathbb{Z}_2$-symmetric global minimizers now exist as shown in \cref{fig:neumannminima}. This is not entirely surprising. There is a cost associated with the pipe splitting and if the optimization problem does not require the flow to leave both outlets, then it is favorable for the flow to exit via one outlet, not both. This is reflected in the resulting cost.

The mesh-independence of the algorithm is investigated in \cref{tab:doublepipeneumanniters}. As before, mesh-independence is observed.

\begin{figure}[ht]
\centering
\includegraphics[width = 0.24\textwidth]{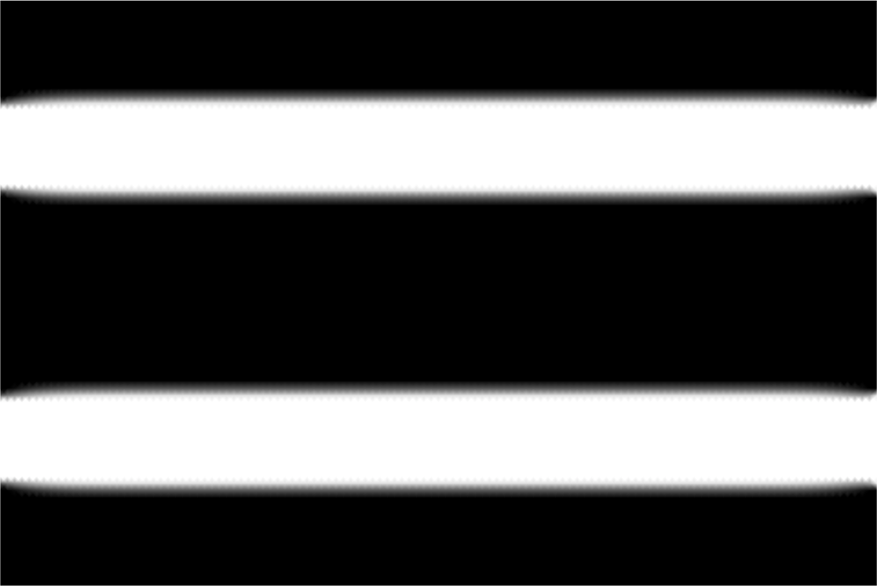}
\includegraphics[width = 0.24\textwidth]{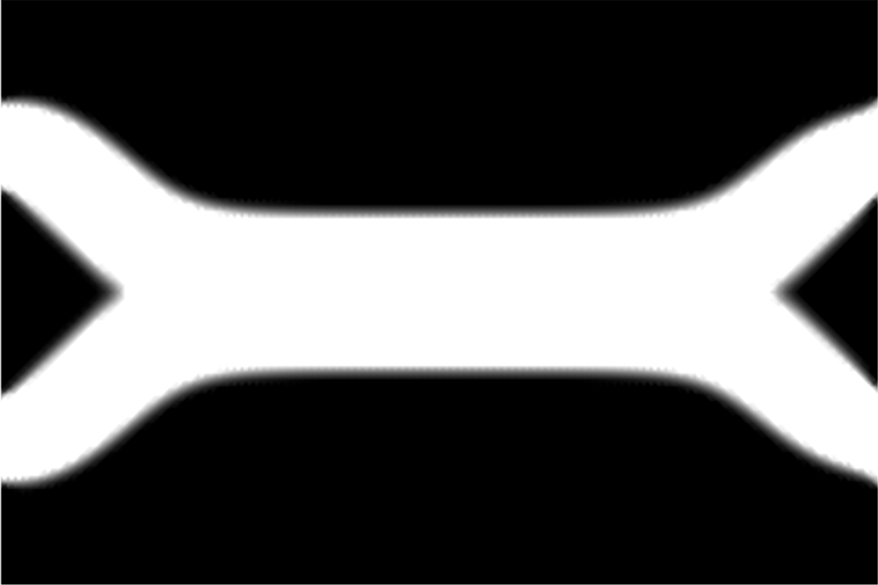}
\includegraphics[width = 0.24\textwidth]{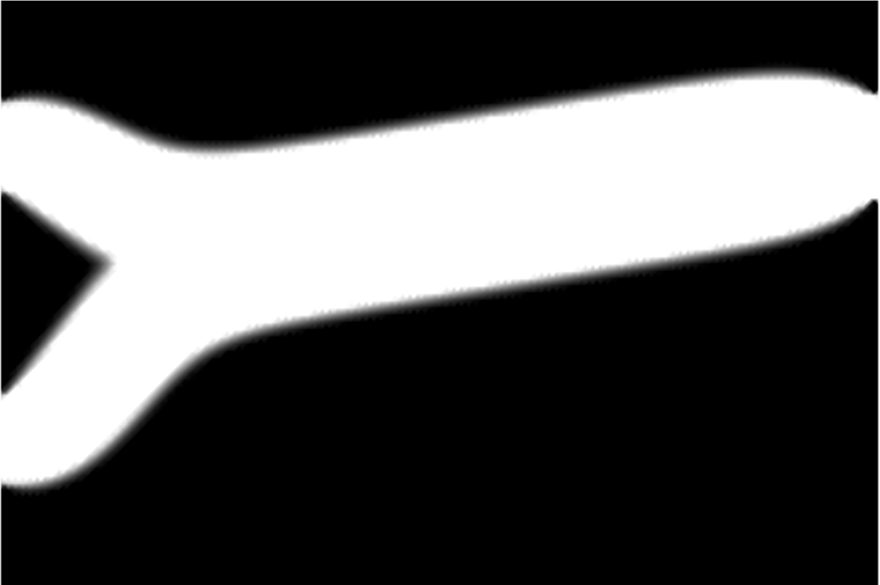}
\includegraphics[width = 0.24\textwidth]{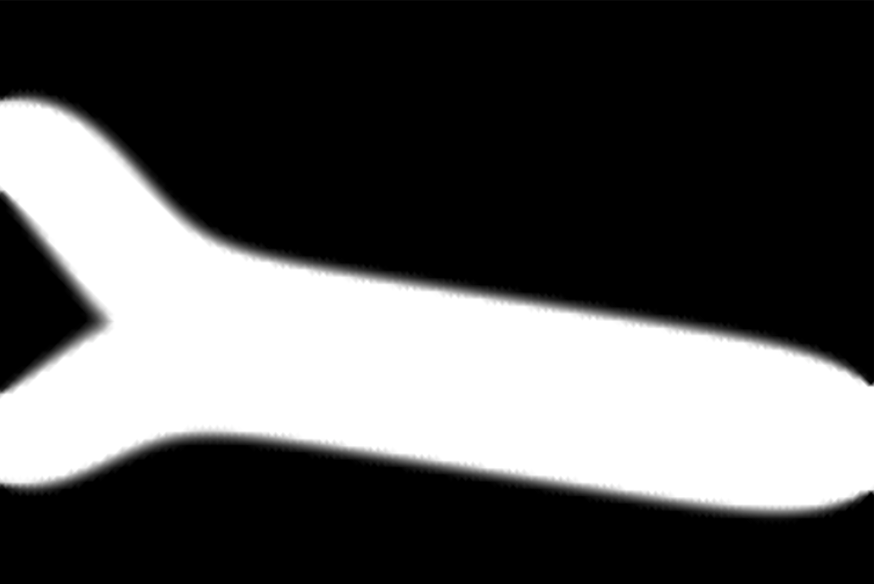}
\caption{The material distribution of two local and two global minimizers of the double-pipe optimization problem with natural boundary conditions on the outlets, instead of Dirichlet conditions, with $h = 0.0125$. Black corresponds to a value of $\rho = 0$ and white corresponds to a value of $\rho = 1$.  From left to right the objective functional values are $J_N = 32.35$, $22.92$, $18.46$, and $18.46$.}
\label{fig:neumannminima}
\end{figure}

\begin{table}[ht]
\small
\begin{center}
\begin{tabular}{ll|lll|lll}
BM Solver & & &Branch 0 &  &   & Branch 1 &  \\ \toprule
$h$&Dofs&Cont. & Defl. & Pred. & Cont. & Defl. & Pred. \\ \midrule
 0.0333& 27,455 &118 		& 0 		     & 53		  & 108 			 & 49 	 & 34  \\
 0.0250& 48,605 &	136 & 0		     &	37  & 	107		&	34 & 37 \\
 0.0125&  193,205&113		& 0		      &35		 &106			&45	 &36  \\ \bottomrule
 &&&&&&& \\
 & & & Branch 2 &  &   & Branch 3 &  \\ \toprule
$h$& Dofs &Cont. & Defl. & Pred. & Cont. & Defl. & Pred. \\ \midrule
 0.0333& 27,455& 166 		& 199 		     & 55		  & 166		 & 149 	 & 55   \\
 0.0250&48,605 & 145 		& 123 		     & 45		  & 145			 & 157 	 & 45   \\
 0.0125&   193,205&128	& 		   151   &	46	 &		128	& 	146 & 46 \\
 \bottomrule
\end{tabular}
\end{center}
\caption{The cumulative total numbers of BM solver iterations required in the continuation, deflation and prediction phases of the double-pipe problem with natural boundary conditions on the outlets.}
\label{tab:doublepipeneumanniters}
\end{table}
\subsection{Roller-type pump}
In this example problem \cite[Sec.\ 2.1.4.4]{Deng2018}, the domain is given by
\begin{align*}
\Omega = (0,1)^2 \backslash \left\{(x,y) \in (0,1)^2 : \left(x-0.5\right)^2 +  \left(y-0.5\right)^2 \leq \left(0.3\right)^2 \right\}.
\end{align*}
The boundary conditions on $\vect{u}$ are given by:
\begin{align*}
\vect{u} =
\begin{cases}
(0, 1-20(x - 0.61)^2)^\top,  & \text{if} \; 0.56 < x < 0.66 \; \text{and} \; y = 0, \\
(1-20(y - 0.95)^2, 0)^\top, & \text{if} \; x = 1 \; \text{and} \;  0.9< y < 1,\\
10/3 (y - 1/2, 1/2 - x)^\top, & \text{if} \; (x-0.5)^2 + (y-0.5)^2  = (0.3)^2,\\
(0,0)^\top, & \text{elsewhere}.
\end{cases}
\end{align*}
These boundary conditions model an inlet on the bottom of the domain and an outlet on the right of the domain with a pump rotating at a constant velocity in the center of the domain where the fluid experiences no-slip boundary conditions. We employ the Taylor--Hood discretization and initialize $\mu_0 = 1000$. Deflation finds the second branch at $\mu = 6.78$.

A global and local minimum of the problem are shown in \cref{fig:rollerminima}. The local minimum chooses to avoid the pump in favor of taking the path with the shortest distance from the inlet to the outlet, while the global minimum exploits the rotation given by the pump. The local minimizer for $q = 1/10$ has areas where $\rho \approx 1/2$, which has an ambiguous physical interpretation. In order to verify whether $\rho$ should be equal to zero or one in such areas, a mixture of grid-sequencing and continuation in $q$ was performed, resulting in the solution shown in \cref{fig:refinedrollerminima}. The mesh-independence of the algorithm is verified in \cref{tab:rollerpumpiters}.
\begin{table}[ht]
\small
\centering
\begin{tabular}{ll|lll|lll}
BM solver& & Branch 0 &  &   & Branch 1 &  \\ \hline
$h_\text{min}$/$h_\text{max}$ & Dofs & Cont. & Defl. & Pred. & Cont. & Defl. & Pred. \\ \hline
0.0258/0.0509&  7388 & 260		& 0		      &55	 &118			&80 	 &35  \\
 0.0127/0.0255& 29,174 & 186		& 0		      &51		 &75			&117 	 &25  \\
0.0064/0.0127&113,096 & 177		& 0		     &46 		  &83  			&99		 &29 \\
\end{tabular}
\caption{The cumulative total numbers of BM solver iterations required in the continuation, deflation and prediction phases of the roller-type pump problem to find the solutions shown in \cref{fig:rollerminima}. The number of iterations are mesh-independent.} 
\label{tab:rollerpumpiters}
\end{table}
\begin{figure}[ht]
\centering
\subfloat[The local (left) and global (right) minimizers, $\rho$.]{\includegraphics[height = 0.3\textwidth]{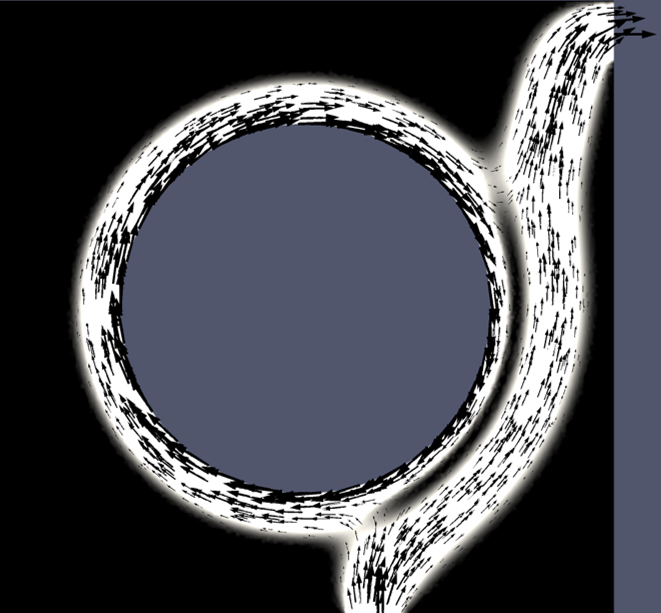}
\includegraphics[height = 0.3\textwidth]{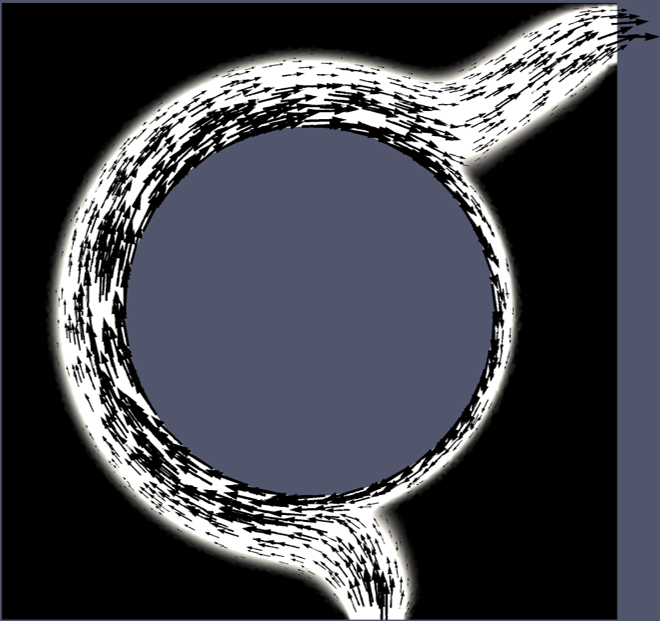}
\label{fig:rollerminima}}
\subfloat[Refined local minimizer.]{\includegraphics[height = 0.3\textwidth]{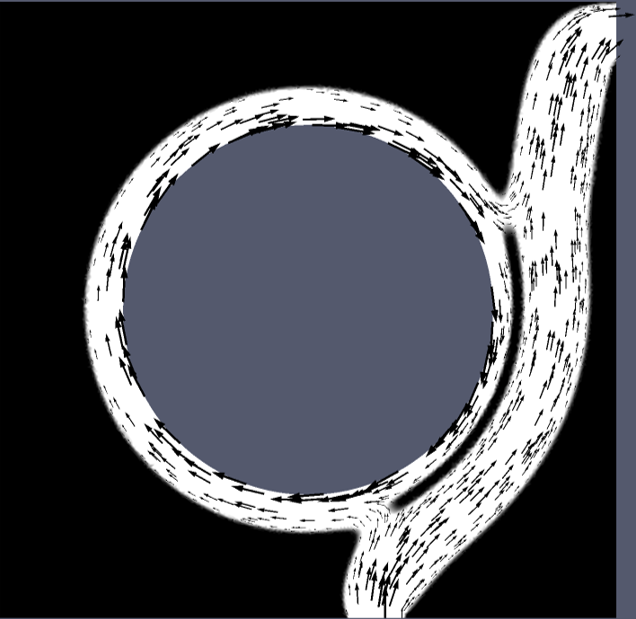}
\label{fig:refinedrollerminima}}
\caption{(a) The material distribution of the local and global minimizers of the roller-type pump optimization problem, with $h_{\text{min}} = 6.4\times 10^{-3}$. Black corresponds to a value of $\rho = 0$ and white corresponds to a value of $\rho = 1$. The gray area is the hole removed from the domain. The arrows indicate the direction and magnitude of the velocity, $\vect{u}$. The values of the objective functional are $J = 26.84$ (left) and $J = 22.67$ (right). (b) A mixture of grid-sequencing of the mesh where $\rho \approx 1/2$ and the continuation of $q$ to larger values was performed on the local minimum of the roller-type pump optimization problem in order to remove areas where $\rho \approx 1/2$. The resulting refined solution has clearly defined areas of $\rho =0$ and $\rho = 1$. Here $h_{\text{min}} =  0.0033$, $q = 0.65$ and $J = 29.17$.}
\end{figure}

\subsection{Five-holes double-pipe with Navier--Stokes}
\label{sec:fiveholesdoublepipe}
We consider the original Borrvall--Petersson double-pipe problem with Dirichlet outflow conditions, but modify the domain to include five small decagonal holes with inscribed radius 0.05 positioned at $(1/2, 1/3)$, $(1/2, 2/3)$, $(1, 1/4)$, $(1, 1/2)$ and $(1, 3/4)$, as shown in \cref{fig:fiveholes}. We further show the flexibility of our method by considering fluid flow constrained by the incompressible Navier--Stokes equations. This is achieved by introducing Lagrange multipliers, $\vect{u}_a \in H^1_0(\Omega)^d$, $p_a \in L^2_0(\Omega)$, and $p_{a,0} \in \mathbb{R}$, to enforce the Navier--Stokes equations. We then define the Lagrangian as
\begin{align}
\begin{split}
&L(\vect{u}, \rho, \vect{u}_a, p, p_a, p_0, p_{a,0}, \lambda) \\
&\indent= J(\vect{u},\rho) - \int_\Omega p \; \text{div}(\vect{u})  \text{d}x - \int_\Omega p_0 p \; \text{d}x  -  \int_\Omega \lambda (\gamma - \rho) \text{d}x  - \int_\Omega p_{a,0} p_a \; \text{d}x\\
&\indent \indent   - \int_\Omega \nu \nabla \vect{u} : \nabla \vect{u}_a + \delta (\vect{u} \cdot \nabla) \vect{u} \cdot \vect{u}_a +\alpha(\rho) \vect{u} \cdot \vect{u}_a - p_a \; \mathrm{div}(\vect{u}_a) \; \mathrm{d}x,
\end{split}
\label{eq:NavierStokesLagrangian}
\end{align}
where $\delta$ denotes the (constant) fluid density. We choose $\nu = 1$ and $\delta = 1$, with other variables equal to those in the original double-pipe problem. We employ the Taylor--Hood discretization and initialize $\mu_0 = 200$. We use feasible tangent prediction and apply an $l^2$-minimizing linesearch in the continuation. 

The holes have the effect of substantially increasing the number of local minima, as shown in \cref{fig:fiveholes-navier}. This example reveals that the number of local minima of a topology optimization problem is not always small and that the deflated barrier method is effective in finding many of them. A small number of solutions found exhibited regions of ambiguity $\rho \approx 1/2$, and underwent grid-sequencing and continuation in $q$ in order to remove these areas. We note that there are more solutions that deflation did \emph{not} find, since there are missing $\mathbb{Z}_2$ symmetric pairs which must also be solutions.
\begin{figure}
\centering
\includegraphics[width = 0.6\textwidth]{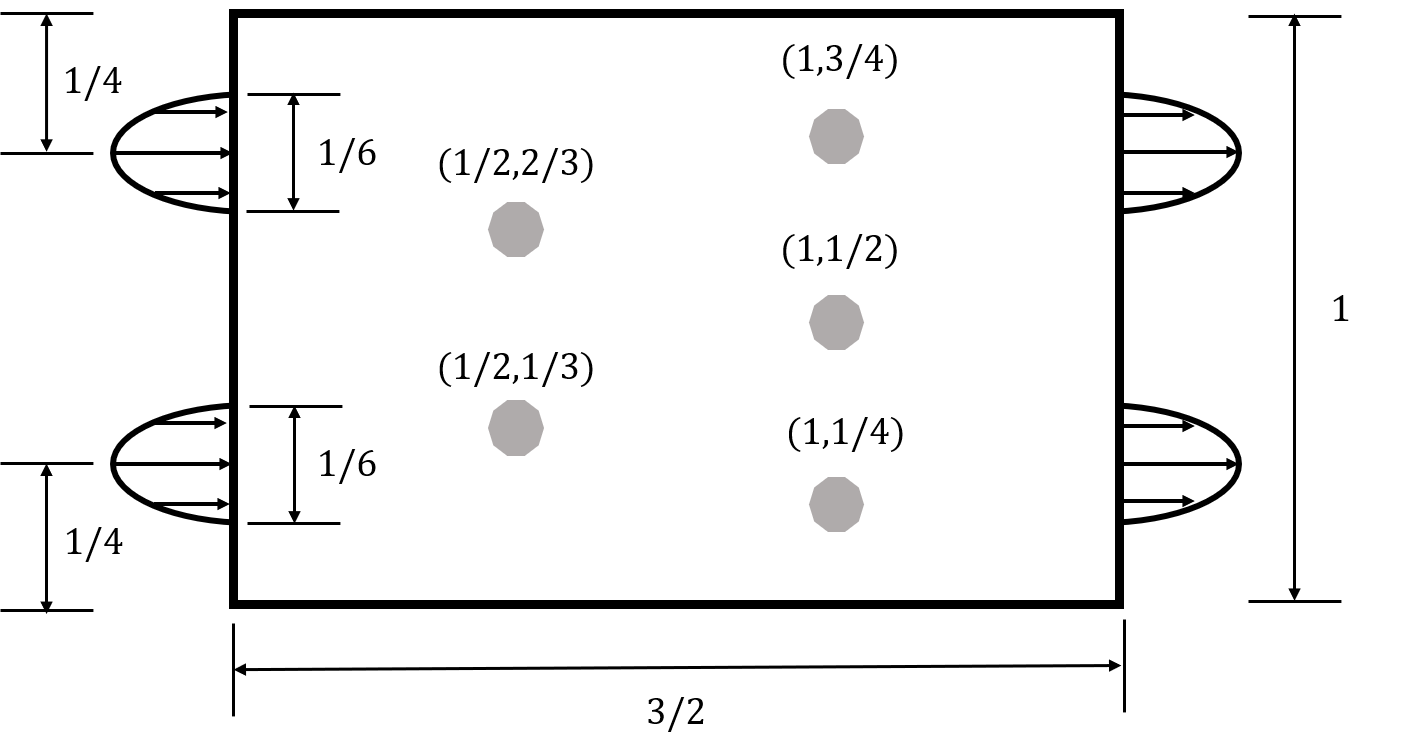}
\caption{Setup of the five-holes double-pipe problem.}
\label{fig:fiveholes}
\end{figure}

\subsection{Cantilever beam}
In this example we use the deflated barrier method to find multiple stationary points of compliance problems. However, due to the lack of regularity of the Lagrange multipliers associated with the box constraints on $\rho$, the solver exhibits mesh-dependent behavior. With each refinement of the mesh, the number of iterations required for the solver to converge increases in an unbounded way. This is difficult to resolve, and appropriate techniques to address this are the subject of ongoing research. Practically, we first run the algorithm on a coarse mesh and then use grid-sequencing to obtain refined solutions.

The two-dimensional cantilever beam optimization problem is to find minimizers of \cref{complianceopt} that satisfy the boundary conditions
\begin{align*}
\sigma \vect{n} & = (0,-1)^\top && \text{on} \; \Gamma_N,\\
\vect{u} &= (0,0)^\top && \text{on} \; \Gamma_D,\\
\sigma\vect{n} & = (0,0)^\top && \text{on} \; \partial \Omega \backslash \{\Gamma_N \cup \Gamma_{D} \},
\end{align*}
with domain $\Omega = (0, 1.5) \times (0,1)$, where
\begin{align*}
\Gamma_D  &= \{ (x,y) \in \partial \Omega : x = 0 \},\\
\Gamma_N  & = \left \{ (x,y) \in \partial \Omega : 0.1 \leq y \leq 0.2,\; x = 1.5 \right \} \cup \left \{ (x,y) \in \partial \Omega : 0.8 \leq y \leq 0.9,\; x = 1.5 \right \}.
\end{align*}
These boundary conditions describe a cantilever clamped to the $y$-axis with two traction forces pulling the cantilever vertically downwards in two places at $x= 1.5$. We use $\mathrm{CG}_1$ finite elements for all variables. We initialize the deflated barrier method at $\mu_0 = 10$ and discover the second branch at $\mu = 4.25\times 10^{-3}$. The two solutions found are shown in \cref{fig:doublecantilever}.
\begin{figure}
\centering
\includegraphics[width = 0.49\textwidth]{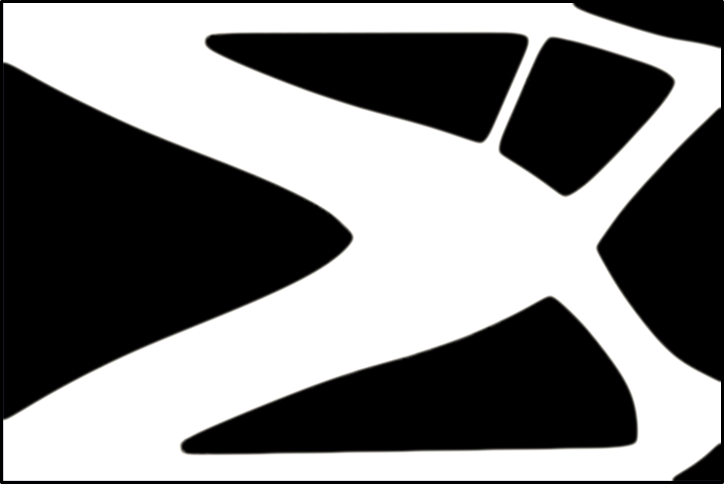}
\includegraphics[width = 0.49\textwidth]{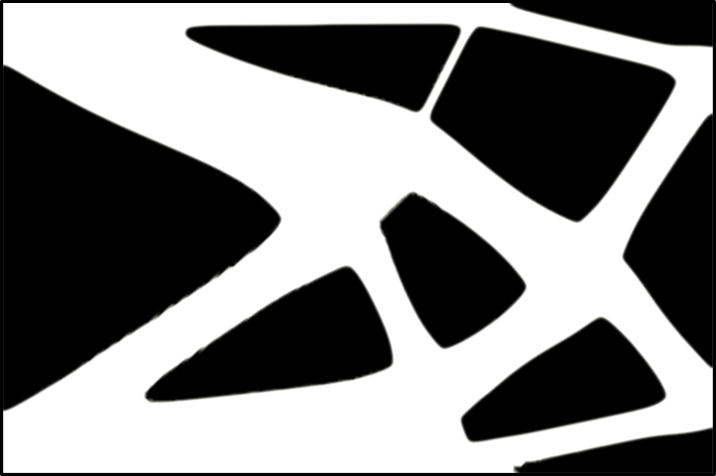}
\caption{The material distribution of two solutions of the cantilever beam. The parameters are $h_\text{min} = 3.56 \times 10^{-3}$, $h_\text{max} = 5.70 \times 10^{-2}$, $\epsilon = 4.40 \times 10^{-3}$, $\beta = 1.8 \times 10^{-4}$, $\gamma = 0.5$, $\epsilon_{\text{SIMP}} = 10^{-5}$, $p_s=3$, and  the Lam\'e coefficients are $\mu_l = 75.38$ and $\lambda_l = 64.62$. $J =6.18\times 10^{-3}$ (left) and $J=6.08\times 10^{-3}$ (right).} \label{fig:doublecantilever}
\end{figure}
\subsection[MBB beam]{Messerschmitt--B\"olkow--Blohm (MBB) beam}
The two-dimensional MBB beam optimization problem is to find minimizers of \cref{complianceopt} that satisfy the boundary conditions
\begin{align*}
\vect{u} \cdot (1,0)^\top & = 0 && \text{on} \; \Gamma_{D_1},\\
\vect{u} \cdot (0,1)^\top & = 0 && \text{on} \; \Gamma_{D_2},\\
\sigma \vect{n} & = (0,-10)^\top && \text{on} \; \Gamma_N,\\
\sigma\vect{n} & = (0,0)^\top && \text{on} \; \partial \Omega \backslash \{\Gamma_N \cup \Gamma_{D_1} \cup \Gamma_{D_2} \} ,\\
\end{align*}
where $\Omega = (0, 3) \times (0,1)$ and
\begin{align*}
\Gamma_{D_1}  &= \{ (x,y) \in \partial \Omega : x = 0 \},\;
\Gamma_{D_2}  = \left\{ (x,y) \in \partial \Omega : y = 0,\; 2.9 \leq  x \leq 3 \right \},\\
\Gamma_N  &= \left \{ (x,y) \in \partial \Omega : y = 1, \;  0 \leq x \leq 0.1 \right \}.
\end{align*}
These boundary conditions describe a half-beam that is fixed horizontally on the $y$-axis and fixed vertically at its bottom right corner on the $x$-axis. There is a boundary force pushing vertically downwards at the top left corner, which represents the middle of the beam when the half-beam is mirrored. We use the same finite element discretization and initialize the deflated barrier method at $\mu_0 = 50$. Deflation discovers the second branch at $\mu = 1.58 \times 10^{-1}$. As in the cantilever problem, the algorithm is mesh-dependent and grid-sequencing is used to find refinements. The two solutions found are shown in  \cref{fig:mbb-beam}.
\begin{figure}[ht]
\centering
\includegraphics[width = 0.49\textwidth]{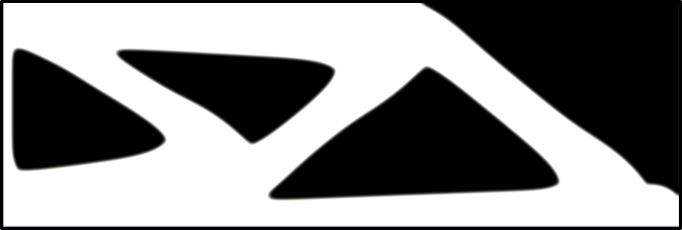}
\includegraphics[width = 0.49\textwidth]{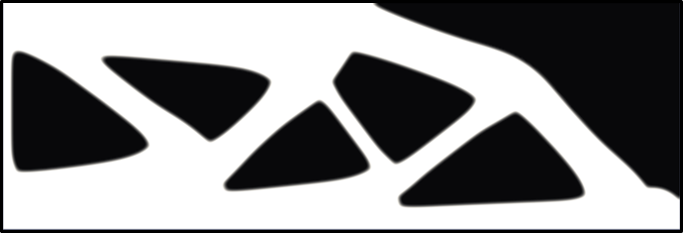}
\caption{The material distribution of two solutions of the MBB beam. The parameters are $h_\text{min} = 7.07 \times 10^{-3}$, $h_\text{max} = 2.83 \times 10^{-2}$, $\epsilon = 1.90 \times 10^{-2}$, $\beta = 9 \times 10^{-3}$, $\gamma = 0.535$, $\epsilon_{\text{SIMP}} = 10^{-5}$, $p_s = 3$, and  the Lam\'e coefficients are $\mu_l = 75.38$ and $\lambda_l = 64.62$. $J=0.723$ (left) and $J=0.681$ (right).} \label{fig:mbb-beam}
\end{figure} 

\section{Conclusions} 
\label{sec:conclusions}
In this work we have developed an algorithm for systemically finding multiple solutions of topology optimization problems. We opted for the density approach, which requires no prior knowledge of the shape or topology of the design. To handle the box constraints on the material distribution $\rho$, we formulate an enlarged-feasible set barrier functional combined with a primal-dual active set solver to ensure the iterates are feasible with respect to the true box constraints. We observe computationally that this approach does not suffer the ill-conditioning or asymptotically infeasible Newton steps that normally hinder primal barrier methods. Furthermore, unlike traditional primal-dual interior point methods, if the Lagrange multipliers of the box constraints in the underlying continuous problem are sufficiently regular, this formulation exhibits mesh-independence. The algorithm successfully found multiple solutions in several problems constrained by the Stokes equations, the Navier--Stokes equations, and the equations of linear elasticity.

\section*{Code availability}
\label{sec:codeavailability}
For reproducibility, the solver and example files to generate the iteration tables and solutions can be found at \url{https://bitbucket.org/papadopoulos/deflatedbarrier/}. The version of the software used in this paper is archived on Zenodo \cite{dbmcode2020}.

\appendix
\section{Benson and Munson's active-set reduced space solver}
\label{sec:BMsolver} 
We show that, in the context of a linear elliptic control problem, if the active and inactive sets of HIK and BM coincide, then the updates calculated for the active and inactive sets are equal. In essence, we show that the algorithms produce iterates that are a \textit{half-step} out of sync, where we define the notion of a half-step below. If the active and inactive sets of BM were redefined to be the same as HIK, then BM would inherit the provably-good convergence properties of HIK. To our knowledge, this is the first analytical result concerning BM. Although the result does not cover the nonlinear case, it might help build an intuitive understanding as to why BM effectively solves the semismooth formulations found in this work. 

Consider the minimization problem
\begin{align}
\min_{y \in L^2(\Omega)} J(y) := \frac{1}{2} (y, Ay)_{L^2(\Omega)} - (f,y)_{L^2(\Omega)} \quad
\text{subject to} \quad y \geq \phi, \label{eq:quadmininf}
\end{align}
where $(\cdot,\cdot)_{L^2(\Omega)}$ denotes the inner product in $L^2(\Omega)$, $f$ and $\phi \in L^2(\Omega)$, and $A \in \mathcal{L}(L^2(\Omega))$ is self-adjoint and coercive. It can be shown there exists a unique solution $y^*$ to \cref{eq:quadmininf} and there exists a Lagrange multiplier $\lambda^* \in L^2(\Omega)$ such that $(y^*,\lambda^*)$ is the unique solution to
\begin{align}
\begin{split}
&Ay - \lambda = f, \\
&y \geq \phi,\; \lambda \geq 0,\; (\lambda, y - \phi)_{L^2(\Omega)} = 0.\label{eq:compinf}
\end{split}
\end{align}
In order to avoid confusion, we denote the iterates generated by HIK by $y_k$ and the iterates generated by BM by $u_k$. The active and inactive sets at iteration $k$, $\mathfrak{A}_k$ and $\mathfrak{I}_k$ in HIK and the active and inactive sets $\mathcal{A}_k$ and $\mathcal{I}_k$ in BM are defined by
\begin{align*}
\mathfrak{A}_k = \{x: \lambda_{k}- (y_{k} - \phi_i)> 0 \}, \;\;&\text{and}\;\;
\mathfrak{I}_k = \{x:  \lambda_{k} - (y_{k}- \phi) \leq 0 \},\\
\mathcal{A}_k = \{x : u_{k} = \phi \; \text{and} \; F(u_{k}) > 0 \}, \;\; &\text{and} \;\;
\mathcal{I}_k = \{x : u_{k} > \phi \; \text{or} \; F(u_{k}) \leq 0 \},
\end{align*} 
where $F(u_k) \in L^2(\Omega)$ is the $L^2$-dual representation of the Fr\'echet derivative of $J(u_k)$.  As in Hinterm\"uller et al.\ \cite[Sec.\ 4]{HintermullerIto2003}, we define $E_{\mathfrak{A}_k}$ the extension-by-zero operator for $L^2(\mathfrak{A}_k)$ to $L^2(\Omega)$-functions, and its adjoint $E_{\mathfrak{A}_k}^*$, the restriction operator of $L^2(\Omega)$ to $L^2(\mathfrak{A}_k)$-functions. We define $E_{\mathfrak{I}_k}$, $E_{\mathfrak{I}_k}^*$, $E_{\mathcal{A}_k}$, $E_{\mathcal{A}_k}^*$, $E_{\mathcal{I}_k}$ and $E_{\mathcal{I}_k}^*$ similarly. We note that all these restriction and extension operators are linear. We now present the infinite-dimensional description of the active-set reduced space strategy (BM).
\begin{enumerate}[label = {(BM\arabic*)}]
\setlength\itemsep{0em}
\item Choose a feasible guess $u_{0} \in L^2(\Omega)$ and set $k$ = 0;
\label{it:BM1}
\item Find $\delta u_k \in L^2(\Omega)$ such that $E_{\mathcal{I}_k}^* A E_{\mathcal{I}_k} E_{\mathcal{I}_k}^* \delta u_{k} = -E_{\mathcal{I}_k}^*(A u_{k} - f)$ \\and  $E_{\mathcal{A}_k}^* \delta u_k=0$;
\label{it:BM2}
\item Set $u_{k+1} =  \pi(u_{k} + \delta u_{k})$ where $\pi$ is the $L^2$-projection onto the constraint, i.e.\ for any given $u \in L^2(\Omega)$, $\pi(u) \in K:=\{ v \in L^2(\Omega): v \geq \phi\}$ satisfies
\begin{align*}
\| u - \pi(u) \|_{L^2(\Omega)} \leq \| u - v \|_{L^2(\Omega)} \; \; \text{for all} \; v \in K.
\end{align*}
\item If convergence is reached, terminate; otherwise set $k \gets k +1$ and go to step \labelcref{it:BM2}.
\label{it:BM4}
\end{enumerate}

\begin{theorem}

Let $y_{k}$ denote the primal variable of HIK at iteration $k$ and let $\delta y_{k}$ denote the update calculated at iteration $k$. Let $\lambda_{k}$ denote the dual variable at iteration $k$. We define half steps such that the active set is updated first, i.e.\ $E_{\mathfrak{A}_k} y_{k+1/2} = E_{\mathfrak{A}_k} y_{k+1}$ and  $E_{\mathfrak{I}_k} y_{k+1/2} = E_{\mathfrak{I}_k} y_{k}$.

Let $u_{k}$ denote the primal variable of BM at iteration $k$ and let $\delta u_{k}$ denote the update calculated at iteration $k$.

Suppose that $\mathcal{A}_k = \mathfrak{A}_k$, $\mathcal{I}_k = \mathfrak{I}_k$ and $E^*_{\mathfrak{I}_k} y_{k} = E^*_{\mathcal{I}_k}u_{k}$. Then the following three equalities hold;
\begin{enumerate}[label = {(E\arabic*)}]
\setlength\itemsep{0em}
\item $y_{k+1/2}= u_{k}$; \label{it:BMtheorem1}
\item $E^*_{\mathfrak{I}_k}\delta y_{k}= E^*_{\mathcal{I}_k} \delta u_{k}$;\label{it:BMtheorem2}
\item $y_{k+3/2}= u_{k+1}$.\label{it:BMtheorem3}
\end{enumerate}
\end{theorem}
\begin{proof}
It is shown in \cite{HintermullerIto2003} that the update for the inactive set of HIK satisfies
\begin{align*}
E^*_{\mathfrak{I}_k}(A \delta y_{k}) &= -E^*_{\mathfrak{I}_k}(A y_{k} - f).
\end{align*}
Expanding the left and right hand sides, we see that
\begin{align*}
E_{\mathfrak{I}_k}^*A E_{\mathfrak{I}_k} E_{\mathfrak{I}_k}^* \delta y_{k} + E_{\mathfrak{I}_k}^* A E_{\mathfrak{A}_k}E_{\mathfrak{A}_k}^* \delta y_{k}= -E_{\mathfrak{I}_k}^* AE_{\mathfrak{I}_k} E_{\mathfrak{I}_k}^* y_{k} - E^*_{\mathfrak{I}_k}A E_{\mathfrak{A}_k} E_{\mathfrak{A}_k}^*y_k+ E_{\mathfrak{I}_k}^*f.
\end{align*}
Subtracting the second term on the left hand side, we see that
\begin{align}
E_{\mathfrak{I}_k}^* A E_{\mathfrak{I}_k} E_{\mathfrak{I}_k}^* \delta y_{k} = -E_{\mathfrak{I}_k}^*AE_{\mathfrak{I}_k}E_{\mathfrak{I}_k}^* y_k - E_{\mathfrak{I}_k}^*AE_{\mathfrak{A}_k} E_{\mathfrak{A}_k}^* (y_{k}+  \delta y_{k})+ E_{\mathfrak{I}_k}^*f. \label{eq:equiv3}
\end{align}
By definition $E_{\mathfrak{A}_k}^*(y+  \delta y_{k}) = E_{\mathfrak{A}_k}^*y_{k+1/2}$ and by assumption $\mathcal{A}_k = \mathfrak{A}_k$, $\mathcal{I}_k = \mathfrak{I}_k$ and $E_{\mathfrak{I}_k}^*y_{k}= E_{\mathfrak{I}_k}^* u_{k}$. Furthermore, since by assumption $\mathcal{A}_k = \mathfrak{A}_k$ and since $E_{\mathfrak{A}_k}^*\delta y_{k} = E_{\mathfrak{A}_k}^*(\phi - y_{k})$ as derived in \cite{HintermullerIto2003} we observe that 
\begin{align}
E_{\mathfrak{A}_k}^* y_{k+1/2} = E_{\mathfrak{A}_k}^* (y_{k} + \phi - y_k) = E_{\mathfrak{A}_k}^*u_{k}. \label{eq:equiv5}
\end{align}
Since, by definition, the first half step in HIK is only an update on the active set, we see that $E_{\mathfrak{I}_k}^*y_{k+1/2}= E_{\mathfrak{I}_k}^*y_{k}= E_{\mathfrak{I}_k}^*u_{k}$.  We therefore have
\begin{align}
y_{k+1/2} = u_{k},
\end{align}
and \labelcref{it:BMtheorem1} holds. From \cref{eq:equiv5}, we can see that \cref{eq:equiv3} is equivalent to
\begin{align}
E_{\mathfrak{I}_k}^*AE_{\mathfrak{I}_k} E_{\mathfrak{I}_k}^*\delta y_{k} &= -E_{\mathfrak{I}_k}^*(A u_{k} - f). \label{eq:equiv4}
\end{align}
We note that \cref{eq:equiv4} is the linear system solved to calculate the update for the inactive set of BM and hence 
\begin{align}
 E_{\mathfrak{I}_k}^*\delta y_{k} =  E_{\mathfrak{I}_k}^*\delta u_{k}. \label{eq:equiv1}
\end{align}
Hence \labelcref{it:BMtheorem2} holds. We now show that $y_{k+3/2} = u_{k}$ by considering four possible cases.

(First case) Consider $C = \mathfrak{I}_{k} \cap \mathfrak{I}_{k+1}$. If $C$ has measure zero, then we are done. Suppose that $|C|>0$. Then since the dual variable is set to zero on the inactive set, we know that $E_C^* \lambda_{k+1} = 0$. Therefore, by definition of $\mathfrak{I}_{k+1}$, we know that $E_C^* y_{k+1} \geq E_C^* \phi$. Hence $E_C^* u_{k} + E_C^* \delta u_{k}\geq E_C^*\phi$ and therefore $E_C^*u_{k+1}= E_C^*\pi(u_{k} + \delta u_{k}) = E_C^*u_{k}+ E_C^*\delta u_{k} = E_C^* y_{k+1}$. The first half step in HIK only changes the active set, hence $E_C^*y_{k+3/2} = E_C^*u_{k+1}$.

(Second case) Consider $C=  \mathfrak{I}_{k} \cap \mathfrak{A}_{k+1}$. If $C$ has measure zero, then we are done. Suppose that $|C|>0$.  Then since the dual variable is set to zero on the inactive set, we know that $E_C^*\lambda_{k+1} = 0$. Therefore, by definition of $\mathfrak{A}_{k+1}$, we know that $E_C^*y_{k+1}< E_C^*\phi$. Hence $E_C^*u_{k} + E_C^*\delta u_{k}< E_C^*\phi$ and therefore $E_C^*u_{k+1} = E_C^*\pi(u_{k} + \delta u_{k}) = E_C^*\phi$. By the half-step update of the active set, $\mathfrak{A}_{k+1}$, $E_C^*y_{k+3/2} = E_C^*\phi$. Hence $E_C^*y_{k+3/2} = E_C^*u_{k+1}$.

(Third case) Consider $C= \mathfrak{A}_{k} \cap \mathfrak{A}_{k+1}$. If $C$ has measure zero, then we are done. Suppose that $|C|>0$.  This implies that $E_C^*y_{k+3/2}= E_C^*\phi$. Since $\mathfrak{A}_{k} = \mathcal{A}_k$, we know that $E_C^*u_{k+1} = E_C^*\phi$. Hence $E_C^*y_{k+3/2} = E_C^*u_{k+1}$.

(Fourth case) Consider $C= \mathfrak{A}_{k} \cap \mathfrak{I}_{k+1}$. If $C$ has measure zero, then we are done. Suppose that $|C|>0$. By definition of $ \mathfrak{A}_k$, this implies that $E_C^*y_{k+1} = E_C^*\phi$. Furthermore, by definition of $\mathfrak{I}_{k+1}$ and since the first half step of HIK only changes the active set, we see that $E_C^*y_{k+3/2} = E_C^*\phi$. By definition of $\mathcal{A}_k$, we know that $E_C^*u_{k+1} = E_C^*\phi$. Hence $E_C^*y_{k+3/2} = E_C^*u_{k+1}$.

From the four cases, we conclude that
\begin{align}
y_{k+3/2} = u_{k+1}. \label{eq:equiv2}
\end{align}
\end{proof}

\section{Feasible tangent predictor} \label{sec:tangentpredictiontask}
Predictor-corrector methods are often used in tracing bifurcation diagrams \cite{Seydel2010}. The idea is that as the parameter of the problem changes, a cheap predictor generates an initial guess for the solution of the system with the new parameter. A corrector method is then used to converge from this initial guess to the true solution. In our context the primal-dual active-set solver is the corrector method. Our feasible tangent predictor method draws inspiration from the usual tangent predictor method, which solves a linear equation to find an initial guess, but applies box constraints to ensure the predicted guess is feasible.

The usual tangent predictor is derived as follows. Consider a Fr\'echet-differentiable equation $F(z^0,\mu^0) = 0$, where $\mu = \mu^0$ is the parameter we wish to vary. Consider a new parameter $\mu = \mu^1$ and let $\delta \mu := \mu^1 - \mu^0$. Furthermore, let $w := (z,\mu)$. The goal is to find $\delta z$ such that $z^0 + \delta z \approx z^1$ where $z^1$ is the solution to
\begin{align}
F(z^1,\mu^1) = 0. \label{eq:Fplus}
\end{align}
A first order approximation of \cref{eq:Fplus} is
\begin{align}
0 = F(z^1,\mu^1) \approx F(z^0,\mu^0) + F'(w)\delta w =  F'_z(z^0,\mu^0) \delta z + F'_\mu(z^0,\mu^0)\delta \mu.
\end{align}
Hence an initial guess, $z_* =  z^0 + \delta z$, can be calculated by solving
\begin{align}
F'_z(z^0,\mu^0)\delta z = -F'_\mu(z^0,\mu^0)\delta \mu,
\end{align}
for $\delta z$. In the context of the deflated barrier method this is equivalent to solving
\begin{align}
(L_{\mu^0}^{\epsilon_{\text{log}}})''|_{\vect{z},\vect{z}}(\vect{z}^0) \delta \vect{z} &+   (L_{\mu^0}^{\epsilon_{\text{log}}})''|_{\vect{z},\mu}(\vect{z}^0) \delta \mu= 0, \label{eq:toptangentpredictor}
\end{align}
for $\delta \vect{z}$. The traditional tangent predictor has no guarantee that $0 \leq \rho^0 + \delta \rho \leq 1$ a.e. To ensure that the initial guess is feasible, we instead transform \cref{eq:toptangentpredictor} into a complementarity problem. Consider the linear operator, $T(\vect{w})$ defined by
\begin{align*}
\langle T(\vect{w}^0), \delta \vect{w} \rangle = (L_{\mu^0}^{\epsilon_{\text{log}}})''|_{\vect{z},\vect{z}}(\vect{z}^0) \delta \vect{z} &+   (L_{\mu^0}^{\epsilon_{\text{log}}})''|_{\vect{z},\mu}(\vect{z}^0) \delta \mu.
\end{align*}
Given sufficient regularity of the dual variable $T(\vect{w})$ and the primal variable $\delta \vect{w}$, we can consider the following complementarity problem,
\begin{alignat}{2}
&\delta \rho(x) = -\rho^0(x) \; &&\text{and} \; T(\vect{w}^0) (x) \geq 0, \label{eq:ftp1}\\
\text{or} \;\;\; &-\rho^0(x) < \delta \rho(x) < 1-\rho^0(x) \;\; &&\text{and} \; T(\vect{w}^0)(x) = 0, \label{eq:ftp2}\\
\text{or} \;\;\; &\delta \rho(x) = 1-\rho^0(x) \; &&\text{and} \;T(\vect{w}^0)(x) \leq 0. \label{eq:ftp3}
\end{alignat}
Solving \cref{eq:ftp1}--\cref{eq:ftp3} constructs a feasible tangent predictor, $\vect{z}_*$. We note that this method does not perform a pointwise projection. For example, in the topology optimization of compliance, where we require the material distribution to live in $H^1(\Omega)$, we are instead performing a $H^1$-projection on the prediction update. In the case where \cref{eq:ftp2} holds a.e.~in $\Omega$, finding the feasible tangent predictor reduces to solving \cref{eq:toptangentpredictor}.

\bibliographystyle{siamplain}
\bibliography{references}
\end{document}

%% file: dbm_shared.tex
\usepackage{lipsum}
\usepackage{amsfonts}
\usepackage{graphicx}
\usepackage{epstopdf}
\usepackage{booktabs}

\usepackage{algpseudocode}
\algnewcommand{\Initialize}[1]{
  \State \textbf{Initialize:}
  \Statex \hspace*{\algorithmicindent}\parbox[t]{.8\linewidth}{\raggedright #1}
}
\renewcommand{\Comment}[2][.5\linewidth]{%
  \leavevmode\hfill\makebox[#1][l]{$\triangleright$~#2}}

\ifpdf
  \DeclareGraphicsExtensions{.eps,.pdf,.png,.jpg}
\else
  \DeclareGraphicsExtensions{.eps}
\fi

\usepackage{enumitem}

\usepackage[caption=false]{subfig}
\usepackage{amssymb}
\newcommand{\vect}[1]{\boldsymbol{#1}}
\newcommand{\vectt}[1]{\boldsymbol{\mathbf{#1}}}

\newsiamremark{remark}{Remark}
\newsiamremark{hypothesis}{Hypothesis}
\crefname{hypothesis}{Hypothesis}{Hypotheses}
\newsiamthm{claim}{Claim}

\headers{Multiple solutions in topology optimization}{I.~P.~A.~Papadopoulos, P.~E.~Farrell and, T.~M.~Surowiec}

\title{Computing multiple solutions of topology optimization problems \thanks{Accepted, January 2021.
\funding{The first author is supported by the EPSRC Centre for Doctoral Training in Partial Differential Equations: Analysis and Applications [grant number  EP/L015811/1] and The MathWorks, Inc. The second author is supported by the Engineering and Physical Sciences Research Council [grant number EP/R029423/1]. The third author is supported by the German Research Foundation [DFG-Grant SU 963/1-1]}}}

\author{Ioannis P.~A.~Papadopoulos\thanks{Mathematical Institute, University of Oxford, Oxford, UK (\email{ioannis.papadopoulos@maths.ox.ac.uk}.}
\and Patrick E.~Farrell\thanks{Mathematical Institute, University of Oxford, Oxford, UK (\email{patrick.farrell@maths.ox.ac.uk}).}
\and Thomas M.~Surowiec \thanks{Department of Mathematics and Computer Science,  Philipps-Universit\"at Marburg, Marburg, DE (\email{surowiec@mathematik.uni-marburg.de}).}}

\usepackage{amsopn}
